\documentclass[11pt,a4paper,reqno]{amsart}

\usepackage{fullpage}

\usepackage{amssymb,dsfont}
\usepackage{tensor}
\usepackage{color}

\DeclareFontFamily{U}{mathb}{\hyphenchar\font45}
\DeclareFontShape{U}{mathb}{m}{n}{
      <5> <6> <7> <8> <9> <10> gen * mathb
      <10.95> mathb10 <12> <14.4> <17.28> <20.74> <24.88> mathb12
      }{}
\DeclareSymbolFont{mathb}{U}{mathb}{m}{n}
\DeclareFontSubstitution{U}{mathb}{m}{n}
\DeclareMathSymbol{\boxvoid}      {2}{mathb}{"6C}

\usepackage[nobysame,alphabetic,initials,msc-links,lite]{amsrefs}

\usepackage[all]{xy}

\DefineSimpleKey{bib}{how}
\renewcommand{\PrintDOI}[1]{%
  \href{http://dx.doi.org/#1}{{\tt DOI:#1}}%
}
\renewcommand{\eprint}[1]{#1}
\BibSpec{book}{%
    +{}  {\PrintPrimary}                {transition}
    +{.} { \PrintDate}                  {date}
    +{.} { \textit}                     {title}
    +{.} { }                            {part}
    +{:} { \textit}                     {subtitle}
    +{,} { \PrintEdition}               {edition}
    +{}  { \PrintEditorsB}              {editor}
    +{,} { \PrintTranslatorsC}          {translator}
    +{,} { \PrintContributions}         {contribution}
    +{,} { }                            {series}
    +{,} { \voltext}                    {volume}
    +{,} { }                            {publisher}
    +{,} { }                            {organization}
    +{,} { }                            {address}
    +{,} { }                            {status}
    +{,} { \PrintDOI}                   {doi}
    +{,} { \PrintISBNs}                 {isbn}
    +{}  { \parenthesize}               {language}
    +{}  { \PrintTranslation}           {translation}
    +{;} { \PrintReprint}               {reprint}
    +{.} { }                            {note}
    +{.} {}                             {transition}
    +{}  {\SentenceSpace \PrintReviews} {review}
}
\BibSpec{article}{%
    +{}  {\PrintAuthors}                {author}
    +{,} { \textit}                     {title}
    +{.} { }                            {part}
    +{:} { \textit}                     {subtitle}
    +{,} { \PrintContributions}         {contribution}
    +{.} { \PrintPartials}              {partial}
    +{,} { }                            {journal}
    +{}  { \textbf}                     {volume}
    +{}  { \PrintDatePV}                {date}
    +{,} { \issuetext}                  {number}
    +{,} { \eprintpages}                {pages}
    +{,} { }                            {status}
    +{,} { \PrintDOI}                   {doi}
    +{,} { \eprint}        {eprint}
    +{}  { \parenthesize}               {language}
    +{}  { \PrintTranslation}           {translation}
    +{;} { \PrintReprint}               {reprint}
    +{.} { }                            {note}
    +{.} {}                             {transition}
    +{}  {\SentenceSpace \PrintReviews} {review}
}
\BibSpec{collection.article}{%
    +{}  {\PrintAuthors}                {author}
    +{,} { \textit}                     {title}
    +{.} { }                            {part}
    +{:} { \textit}                     {subtitle}
    +{,} { \PrintContributions}         {contribution}
    +{,} { \PrintConference}            {conference}
    +{}  {\PrintBook}                   {book}
    +{,} { }                            {booktitle}
    +{,} { \PrintDateB}                 {date}
    +{,} { pp.~}                        {pages}
    +{,} { }                            {publisher}
    +{,} { }                            {organization}
    +{,} { }                            {address}
    +{,} { }                            {status}
    +{,} { \PrintDOI}                   {doi}
    +{,} { \eprint}        {eprint}
    +{}  { \parenthesize}               {language}
    +{}  { \PrintTranslation}           {translation}
    +{;} { \PrintReprint}               {reprint}
    +{.} { }                            {note}
    +{.} {}                             {transition}
    +{}  {\SentenceSpace \PrintReviews} {review}
}
\BibSpec{misc}{%
  +{}{\PrintAuthors}  {author}
  +{,}{ \textit}      {title}
  +{.}{ }             {how}
  +{}{ \parenthesize} {date}
  +{,} { available at \eprint}        {eprint}
  +{,}{ available at \url}{url}
  +{,}{ }             {note}
  +{.}{}              {transition}
}

\numberwithin{equation}{section}

\theoremstyle{plain}
\newtheorem{theorem}{Theorem}[section]
\newtheorem{proposition}[theorem]{Proposition}
\newtheorem{lemma}[theorem]{Lemma}
\newtheorem{corollary}[theorem]{Corollary}

\theoremstyle{definition}
\newtheorem{definition}[theorem]{Definition}
\newtheorem{example}[theorem]{Example}

\theoremstyle{remark}
\newtheorem{remark}[theorem]{Remark}

\mathchardef\mhyph="2D

\newcommand{\bp}{\begin{proof}}
\newcommand{\ep}{\end{proof}}
\newcommand{\eps}{\varepsilon}
\newcommand{\C}{\mathbb{C}}
\newcommand{\Z}{\mathbb{Z}}
\newcommand{\R}{\mathbb{R}}

\newcommand{\co}{\mathrm{co}}

\newcommand{\cat}[1]{\mathcal{#1}}
\newcommand{\A}{\cat{A}}
\newcommand{\CC}{\cat{C}}
\newcommand{\PP}{\cat{P}}
\newcommand{\DD}{\cat{D}}

\newcommand{\tCC}{\tilde{\CC}}
\newcommand{\tDD}{\tilde{\DD}}
\newcommand{\un}{\mathds{1}}
\newcommand{\IdSet}{\mathfrak{I}}
\newcommand{\Pol}{\mathcal{O}}
\newcommand{\GL}{\mathrm{GL}}
\newcommand{\SL}{\mathrm{SL}}

\DeclareMathOperator{\Aut}{Aut}
\DeclareMathOperator{\Ch}{Ch}

\DeclareMathOperator{\Comod}{Comod}
\DeclareMathOperator{\Corep}{Corep}
\DeclareMathOperator{\End}{End}

\DeclareMathOperator{\id}{id}
\DeclareMathOperator{\Id}{Id}
\DeclareMathOperator{\Irr}{Irr}

\DeclareMathOperator{\Rep}{Rep}
\DeclareMathOperator{\tr}{tr}

\begin{document}

\title[Graded twisting]{Graded twisting of comodule algebras and module categories}

\author[J. Bichon]{Julien Bichon}
\email{Julien.Bichon@math.univ-bpclermont.fr}
\address{Laboratoire de Math\'{e}matiques, Universit\'{e} Blaise Pascal, Campus universitaire des C\'{e}zeaux, 3 place Vasarely, 63178 Aubi\`{e}re Cedex, France}

\author[S. Neshveyev]{Sergey Neshveyev}
\email{sergeyn@math.uio.no}
\address{Department of Mathematics, University of Oslo, P.O. Box 1053
Blindern, NO-0316 Oslo, Norway}
\thanks{The research leading to these results has received funding
from the European Research Council under the European Union's Seventh
Framework Programme (FP/2007-2013)
/ ERC Grant Agreement no. 307663
}

\author[M. Yamashita]{Makoto Yamashita}
\email{yamashita.makoto@ocha.ac.jp}
\address{Department of Mathematics, Ochanomizu University, Otsuka
2-1-1, Bunkyo, 112-8610 Tokyo, Japan}
\thanks{Supported by JSPS KAKENHI Grant Number 25800058}

\date{April 11, 2016; new version January 22, 2017; minor revision June 15, 2017}                                           

\begin{abstract}
Continuing our previous work on graded twisting of Hopf algebras and monoidal categories, we introduce a graded twisting construction for equivariant comodule algebras and module categories. As an example we study actions of quantum subgroups of $G\subset\SL_{-1}(2)$ on $K_{-1}[x,y]$ and show that in most cases the corresponding invariant rings $K_{-1}[x,y]^G$ are invariant rings  $K[x,y]^{G'}$ for the action of a classical subgroup $G'\subset \SL(2)$. As another example we study Poisson boundaries of graded twisted categories and show that under the assumption of weak amenability they are graded twistings of the Poisson boundaries.
\end{abstract}

\maketitle

\section*{Introduction}

Let $A$, $B$ be Hopf algebras, and assume that $B$ is a $2$-cocycle twisting of $A$. Then we have a monoidal equivalence between the categories of left comodules
$$\mathcal F\colon \Comod(A) \simeq^\otimes \Comod(B)$$
It is a simple observation  that $\mathcal F$ induces an equivalence between the respective categories of comodule algebras, and that if $R$ is an $A$-comodule algebra,  we have an isomorphism between the fixed point algebras
$$^{\co B}\!\mathcal F(R) \simeq {^{\co A}\!R}$$
In other words, the invariant theory of $A$ determines that of $B$.

One of the main goals of this paper is to give a version of this   principle when $B$ is a \textsl{graded twisting} of $A$~\cite{MR3580173}. In this case there is also an  equivalence  $\Comod(A) \simeq \Comod(B)$, but it is not monoidal in general, so we cannot expect  a correspondence between the categories of comodule algebras. We will show that nevertheless for $A$-comodule algebras with some additional structure we do get $B$-comodule algebras, and this construction preserves the fixed point algebras. We also develop a categorical version of this construction and consider several examples. Specifically, the contents and main results of the paper are as follows.

\smallskip

Section~\ref{sec:preliminaries} consists of recollections and preliminaries on graded twistings, both in Hopf algebraic and categorical settings. Here we also point out that the underlying product of the graded twisting of a Hopf algebra is a particular case of the twisted product from \cite{MR1367080}. In fact, some of our basic observations easily extend to the setting of \cite{MR1367080}, and we indicate this in several remarks.

\smallskip

In Section~\ref{sec:twist-groups-acti} we study quotients of graded twistings of commutative Hopf algebras of functions on affine algebraic groups, giving a complete description of such noncommutative quotients when the group $\Gamma$ defining the grading is finite of prime order. This complements our results on graded twistings of compact groups in~\cite{MR3580173}, but while in~\cite{MR3580173} we used C$^*$-algebraic tools, here the proofs are purely Hopf algebraic.

\smallskip

In Section~\ref{sec:twist-comod-algebr} we introduce graded twisting of $\Gamma$-equivariant comodule algebras. When $\Gamma$ is abelian we show that this construction defines an equivalence between the categories of such algebras. As an example we study actions of quantum subgroups of $\SL_{-1}(2)$ on $K_{-1}[x,y]$, a situation that cannot be covered by the $2$-cocycle twisting framework. Our results are refinements of some of those in \cite{arXiv:1303.7203}, where one of the conclusions is that the invariant rings $K_{-1}[x,y]^G$ for actions of finite quantum subgroups $G \subset \SL_{-1}(2)$ share similar homological properties with invariant rings $K[x,y]^G$ for classical finite group actions: indeed, using the results of Section~\ref{sec:twist-groups-acti}, we show that in most cases these invariant rings $K_{-1}[x,y]^G$ are in fact invariant rings  $K[x,y]^{G'}$ for the action of a classical subgroup $G'\subset \SL(2)$.

\smallskip

In Section~\ref{sec:categorical-picture} we define a categorical counterpart of the construction of Section~\ref{sec:twist-comod-algebr}, the graded twisting of equivariant module categories. In the categorical setting the construction is actually almost tautological, but we show that it is indeed the right analogue of that in Section~\ref{sec:twist-comod-algebr}, the two being related to each other by a Tannaka--Krein type duality result.

\smallskip

In Section~\ref{sec:poisson-boundary} we study the graded twistings of module categories provided by the Poisson boundary theory. In fact, these are more than just module categories, as the module category structure is defined by a tensor functor, so in principle such a study could have been carried out already in~\cite{MR3580173}. Our main result here is that under the assumption of weak amenability the constructions of Poisson boundaries and graded twistings commute (up to equivalence).

\smallskip
\paragraph{\bf Acknowledgement} We would like to thank C{\'e}sar Galindo for comments on the first version of the paper, and the referee for having drawn our attention to the references \cites{MR1367080,MR1386496} and for his careful reading which helped us improve the presentation.

\section{Preliminaries}
\label{sec:preliminaries}

We mainly follow the conventions of~\cite{MR3580173}. Thus, $K$ stands for a field of characteristic zero, $\Gamma$ is a discrete group, and $A$ is a Hopf algebra. All algebras, linear categories, vector spaces and so on are considered over $K$. The group algebra of $\Gamma$ (over $K$) is denoted by $K \Gamma$. When we talk about $*$-algebras, it is always assumed that $K = \C$.  We also assume that `the' category of $K$-vector spaces has a strict tensor product. We denote the category of left $A$-comodules by $\Comod(A)$. Identity maps and identity morphisms are tacitly denoted by $\iota$.


\subsection{Graded Hopf algebras}\label{subsec:graded}

Our starting point is the following structure on $A$, which can be defined in two equivalent ways~\cite{MR3580173}:
\begin{enumerate}
\item \label{it:gr-cond-1}
a  \emph{cocentral} Hopf algebra homomorphism $p\colon A \rightarrow K\Gamma$, i.e., $p(a_{(1)})\otimes a_{(2)}=p(a_{(2)})\otimes a_{(1)}$ for all $a \in A$;
\item \label{it:gr-cond-2}
a direct sum decomposition $A=\bigoplus_{g \in \Gamma} A_g$ such that $A_gA_h\subset A_{gh}$ and $\Delta(A_g)\subset A_{g}\otimes A_g$  for all $g,h \in \Gamma$.
\end{enumerate}
Namely, given~\eqref{it:gr-cond-1}, the grading is defined by
$$
A_g=\{a \in A \ | p(a_{(1)})\otimes a_{(2)}=g \otimes a \}=\{a \in A \ | a_{(1)}\otimes p(a_{(2)})=a \otimes g \},
$$
while if~\eqref{it:gr-cond-2} is given, the map $p$ is defined by $p(a)=\varepsilon(a)g$ for every $a \in A_g$. Note that we always have  $1 \in A_e$ and $S(A_g) \subset A_{g^{-1}}$ for the antipode $S$. In particular, $A_e$ is a Hopf subalgebra of $A$. The map $p$ is surjective if and only if $A_g \not = \{0\}$ for every $g \in \Gamma$.

This structure implies a number of properties of $A$. To formulate a precise result let us first recall the following definition.

\begin{definition}[\citelist{\cite{MR1228767}\cite{MR1334152}}]
A sequence  of Hopf algebra homomorphisms
\begin{equation*}
K \to B \overset{i}\to A \overset{p}\to L \to
K\end{equation*}
is said to be \emph{exact} if the following
conditions hold:
\begin{enumerate}
\item $i$ is injective, $p$ is surjective, and $p i(b) = \varepsilon(b) 1$, for every $b \in B$;
\item $\ker p = A i(B)^+ =i(B)^+ A$, where $i(B)^+=i(B)\cap\ker(\varepsilon)$;
\item the image of $i$ is equal to the coinvariants of $L$, that is,
$$
i(B) = A^{{\rm co} L} = \{ a \in A \mid (\id \otimes p)\Delta(a) = a \otimes 1 \} = {^{{\rm co} L}A} = \{ a \in A \mid (p \otimes \id)\Delta(a) = 1 \otimes a \}.
$$
\end{enumerate}
\end{definition}
Note that the condition $p i = \varepsilon(\cdot) 1$ follows from the other ones.

\begin{proposition}\label{pbasic}
Assume we are given a surjective cocentral Hopf algebra homomorphism $p\colon A \rightarrow K\Gamma$. Then
\begin{enumerate}
\item the grading on $A$ is strong, i.e., $A_gA_h=A_{gh}$ for all $g,h\in\Gamma$; we also have $A^+_gA_h=A_gA_h^+=A^+_{gh}$;
\item $A_g$ is a finitely generated projective left and right $A_e$-module for every $g\in\Gamma$;
\item $A$ is a faithfully flat left and right $A_e$-module, as well as a faithfully coflat left and right $K\Gamma$-comodule;
\item we have an exact sequence of Hopf algebras $K \to A_e \to A \to K \Gamma \to K$.
\end{enumerate}
\end{proposition}

\bp

(i) Let us fix $g,h \in \Gamma$. Since $A_{h^{-1}}$ is a nonzero subcoalgebra of $A$, we can take $b\in A_{h^{-1}}$ such that $\eps(b)=1$. Then for each $a\in A_{gh}$ we have
$$
a=ab_{(1)}S(b_{(2)})\in A_g A_h.
$$
The same argument shows that if $a\in A^+_{gh}$, then $a\in A^+_gA_h$. Similarly one checks that $A^+_{gh}=A_gA^+_h$.

\smallskip

(ii) This is a general property of strong gradings: if we choose $x_i\in A_g$ and $y_i\in A_{g^{-1}}$ such that $\sum^n_{i=1}y_ix_i=1$, then we can define a left $A_e$-module map $A_g\to A^n_e$ by $a\mapsto (ay_i)^n_{i=1}$ and its left inverse $A_e^n\to A_g$ by $(a_i)^n_{i=1}\mapsto\sum_ia_ix_i$.

\smallskip

(iii) The faithful flatness follows immediately from (ii). The faithful coflatness follows from the cosemisimplicity of $K\Gamma$, but is  also easy to check directly, since if $V$ is a left $K\Gamma$-comodule, then the cotensor product $A \boxvoid_{K\Gamma} V$ is equal to $\bigoplus_g A_g \otimes V_g$ with respect to the induced $\Gamma$-gradings.

\smallskip

(iv) We only have to check that $\ker p=AA_e^+=A_e^+A$. As $\ker p=\bigoplus_g A_g^+$, this follows from~(i).
\ep

\subsection{Graded twisting of Hopf algebras}

Let us now recall the graded twisting construction of Hopf algebras introduced in~\cite{MR3580173}.

\begin{definition}[\cite{MR3580173}]
An \emph{invariant cocentral action} of $\Gamma$ on $A$ is a pair $(p, \alpha)$ where
\begin{itemize}
\item $p\colon A \rightarrow K\Gamma$ is a cocentral Hopf algebra homomorphism,
\item $\alpha\colon \Gamma \rightarrow \Aut(A)$ is an action of $\Gamma$ by Hopf algebra automorphisms on $A$, with $p\alpha_g=p$ for all $g \in \Gamma$.
\end{itemize}
In terms of the $\Gamma$-gradings, the last condition is equivalent to $\alpha_g(A_h)= A_h$ for all $g,h \in \Gamma$. When $(p,\alpha)$ is such an action, the \emph{graded twisting} $A^{t,\alpha}$ of $A$ is the Hopf algebra
$$
A^{t, \alpha}=\sum_{g \in \Gamma} A_g \otimes g\subset A \rtimes \Gamma,
$$
equipped with the algebra structure of the crossed product $A \rtimes\Gamma$, which we identify with  $A \otimes K\Gamma$ as a linear space, and with the coalgebra structure induced by the linear isomorphism
\begin{equation}\label{eq:def-of-j}
 j\colon A \to A ^{t,\alpha}, \quad \sum_g a_g \mapsto \sum_g a_g \otimes g.
\end{equation}
\end{definition}

Note that $\tilde{p}= p \otimes \varepsilon\colon A^{t,\alpha}\to K\Gamma$ is a cocentral Hopf algebra homomorphism, which is surjective when $p$ is. However, the maps $\tilde\alpha_g=(\alpha_g \otimes \iota)|_{A^{t,\alpha}}$ are not necessarily algebra automorphisms unless~$\Gamma$ is abelian, see the discussion in~\cite{MR3580173}*{Section~2.3}.

\begin{remark}\label{rem:previoustwist}
The algebra structure on a graded twisting as above is a very special case of the one constructed in \cites{MR1367080,MR1386496}. Indeed, following  the setting in \cite{MR1367080}, consider an algebra~$A$ graded by a group $\Gamma$ as in Subsection \ref{subsec:graded}, and a map $\tau$ from $\Gamma$ into the group of graded linear automorphisms of $A$ such that, for every $g \in \Gamma$, $\tau_g$ is unit preserving, $\tau_1= \iota_A$, and such that, for all $g,h \in \Gamma$, $a \in A_h$, $b \in A$, we have
$$\tau_g(a\tau_h(b))=\tau_{g}(a)\tau_{gh}(b).$$
Following \cite{MR1367080}, one defines a new product on $A$ by
$$a \cdot_{\tau} b= a\tau_g(b), \ a\in A_g,\ b \in A,$$
and this produces a new twisted associative algebra, denoted $A^\tau$. Now, if we assume in addition that $A$ is a Hopf algebra as in Subsection~\ref{subsec:graded} and the maps $\tau_g$ are coalgebra automorphisms, one can easily check that $A^\tau$, with the original coalgebra structure on $A$, becomes a Hopf algebra.

It is straightforward that, starting from an invariant cocentral action $(p,\alpha)$, the map $\alpha$
satisfies the $\tau$-conditions above, and that the graded twisting Hopf algebra $A^{t, \alpha}$ is isomorphic with the $A^\alpha$ above. Therefore the $\tau$-setting generalizes the setting of invariant cocentral actions. Some of the basic results on graded twisting can then be extended to $A^\tau$, as will be indicated in several places. At the same time most of our considerations in the subsequent sections do not seem to have simple generalizations to this setting. Nevertheless, one clear advantage of the $\tau$-setting, pointed to us by the referee, is that it puts the original Hopf algebra and the twisted one on an equal footing, since one can come back to the original $A$ by twisting again using the map $g \mapsto \tau_g^{-1}$. This observation overcomes the difficulty that the maps $\tilde\alpha_g=(\alpha_g \otimes \iota)|_{A^{t,\alpha}}$ are not necessarily algebra automorphisms unless~$\Gamma$ is abelian, and would have simplified  and clarified the discussion in~\cite{MR3580173}*{Section~2.3}.

An even more general construction arises from \cite{MR1386496}. Assume we are given a Hopf algebra~$H$, a (right) $H$-comodule algebra $A$ and a convolution invertible linear map $\tau$ from $H$ into the space of linear endomorphisms of $A$. A new product on $A$ is defined by
$$a \cdot_\tau b = a_{(0)}\tau_{a_{(1)}}(b), \ a,b \in A.$$
Conditions that ensure that the new product is associative (with $1$ as unit) are given in \cite{MR1386496}*{Theorem 1.1}. Assuming in addition that $A$ is a Hopf algebra, it is not difficult to write down axioms that ensure that the new algebra $A^\tau$ is a Hopf algebra (the coalgebra structure being the original one of $A$). While it is hard to imagine that there are meaningful analogues of most of the results of the present paper to this setting, this construction is probably worth studying for other purposes. For example, in view of the results of \cite{MR1386496}*{Section 3}, this general twisting procedure should enable one to treat the crossed product Hopf algebras as studied in \cites{MR3169407,MR2802546}.
\end{remark}

Coming back to a Hopf algebra $A$ endowed with an invariant cocentral action $(p, \alpha)$ of the group $\Gamma$, since the coalgebras $A$ and $A^{t, \alpha}$ are isomorphic, there is an  equivalence of comodule categories
\begin{equation}\label{eq:comod-eqv-for-twist}
F\colon \Comod(A) \to \Comod(A^{t, \alpha}), \quad V \mapsto V^{t,\alpha} = \bigoplus_{g \in \Gamma} V_g \otimes g,
\end{equation}
as $K$-linear categories, see~\cite{MR3580173}*{Section~2} for the case of finite dimensional comodules. Here, $V_g$ denotes the $g$-homogeneous component of $V$ for the $\Gamma$-grading induced by the coaction of $K \Gamma$ given by the composition of $A$-coaction and $p$. The $A^{t,\alpha}$-comodule structure of $V^{t,\alpha}$ is defined by $v \otimes g \mapsto v_{(-1)} \otimes g \otimes v_{(0)} \otimes g$ for $v \in V_g$.

In general this equivalence is not an equivalence of monoidal categories, but it is a quasi-monoidal equivalence when the cocentral action is \emph{almost adjoint} in the sense of \cite{MR3580173}. In Section~\ref{sec:twist-comod-algebr} we will see that for appropriate categories of comodules $F$ does become a monoidal equivalence.

\subsection{Graded twisting of monoidal categories}\label{ssec:cat-twist}

Let us now recall the categorical counterpart of the above construction~\cite{MR3580173}.

Let $\CC$ be a $\Gamma$-graded $K$-linear monoidal category. Thus, we are given full subcategories~$\CC_g$ for $g \in \Gamma$ such that any object $X$ in $\CC$ admits a unique (up to isomorphism) decomposition $X \simeq \bigoplus_{g \in \Gamma} X_g$ with $X_g \in \CC_g$ (and $X_g=0$ for all but a finite number of $g$'s), there are no nonzero morphisms between the objects in $\CC_g$ and $\CC_h$ for $g\ne h$, $\un \in \CC_e$, and the monoidal structure satisfies $X \otimes Y \in \CC_{g h}$ for all homogeneous objects $X \in \CC_g$ and $Y \in \CC_h$.

Consider now a weak action of $\Gamma$ on $\CC$, that is, a monoidal functor $\alpha\colon\underline{\Gamma}\to\underline{\Aut}^\otimes(\CC)$, where~$\underline{\Gamma}$ is the monoidal category with objects the elements of $\Gamma$, no nontrivial morphisms, and with the tensor structure given by the product in $\Gamma$, and where $\underline{\Aut}^\otimes(\CC)$ is the monoidal category of monoidal autoequivalences of $\CC$, with the tensor structure given by the composition of monoidal functors. More concretely, this means that we are given monoidal autoequivalences $(\alpha^g,\alpha^g_2,\alpha^g_0)$ of $\CC$ for each $g \in \Gamma$, which we usually denote simply by $\alpha^g$, where
\begin{itemize}
\item  $\alpha^g$ is an autoequivalence of $\CC$ as a $K$-linear category,
\item $\alpha^g_2\colon \alpha^g(X) \otimes \alpha^g(Y) \to \alpha^g(X \otimes Y)$ is a natural family of isomorphisms for $X, Y$ in $\CC$,
\item $\alpha^g_0 \colon \un \to \alpha^g(\un)$ is an isomorphism,
\end{itemize}
which satisfy the standard set of compatibility conditions for associator and unit, and natural monoidal isomorphisms  $\eta^{g,h}=(\eta^{g,h}_X\colon\alpha^g\alpha^h(X)\to\alpha^{gh}(X))_{X\in\CC}$ from $\alpha^g \alpha^h$ to $\alpha^{g h}$ such that $\alpha^e\simeq\Id_\CC$ and
\begin{equation*}
\label{eq:assoc-eta-cond}
\eta^{g, h k}_X \alpha^g(\eta^{h,k}_X) = \eta^{g h, k}_X \eta^{g,h}_{\alpha^k(X)}
\end{equation*}
as morphisms from $\alpha^g \alpha^h \alpha^k(X)$ to $\alpha^{g h k}(X)$. Replacing $\underline{\Gamma}\to\underline{\Aut}^\otimes(\CC)$ by a naturally monoidally isomorphic functor we may and, when convenient, will  assume that
\begin{equation}\label{eq:weakact1}
\alpha^e=\Id_\CC,\ \ \eta^{e,g}_X=\eta^{g,e}_X=\iota.
\end{equation}
If $\CC$ is strict, then similarly we may assume that
\begin{equation}\label{eq:weakact2}
\alpha^g(\un)=\un,\ \ \alpha^g_2(\un,X)=\alpha^g_2(X,\un)=\iota.
\end{equation}
Note that we then also have
\begin{equation}\label{eq:weakact3}
\alpha^g_0=\iota,\ \ \eta^{g,h}_\un=\iota.
\end{equation}

A weak action $(\alpha,\eta)$ is called \emph{invariant} if every $\alpha^g$ preserves the homogeneous subcategories~$\CC_h$ for all $h\in \Gamma$. Given such an action, we denote by $\CC^{t,(\alpha, \eta)}$, or simply by $\CC^{t,\alpha}$, the full monoidal $K$-linear subcategory of $\CC \rtimes_{\alpha,\eta} \Gamma$ (called the \emph{semidirect product} in \cite{MR1815142}) obtained by taking direct sums of the objects $X \boxtimes g$ for $g \in \Gamma$ and $X \in \CC_g$, and call $\CC^{t,\alpha}$ the \emph{graded twisting of $\CC$}.

By construction, $\CC^{t,\alpha}$ is equivalent to $\CC$ as a $\Gamma$-graded $K$-linear category. Identifying $\CC$ and $\CC^{t,\alpha}$ as $K$-linear categories, we may express the twisted monoidal structure as a bifunctor on~$\CC$ given by $X \otimes_{\alpha} Y = X \otimes \alpha^g(Y)$ for $X \in \CC_g$.

\begin{remark}
It is  possible to define a categorical analogue of $\tau$-twisting discussed in Remark~\ref{rem:previoustwist}. Let $\CC$ be a $\Gamma$-graded $K$-linear monoidal category. Assume we are given autoequivalences $\tau_g$, $g\in\Gamma$, of $\CC$ as a graded linear category such that $\tau_g(\un)=\un$ and $\tau_e=\Id_\CC$, and natural isomorphisms
$$
\psi_{g,h}\colon\tau_g(X\otimes\tau_h(Y))\to\tau_g(X)\otimes\tau_{gh}(Y),\ X\in\CC_h,\ Y\in\CC.
$$
We then define a new tensor structure on $\CC$ such that
$$
X\otimes_\tau Y=X\otimes\tau_g(Y),\ X\in\CC_g,\ Y\in\CC.
$$
We then need to impose an additional requirement on $\tau$ and $\psi$, a detailed formulation of which we leave to the reader, saying that this product together with associativity morphisms given by the compositions
$$
(X\otimes\tau_g(Y_h))\otimes\tau_{gh}(Z)\xrightarrow{\Phi} X\otimes (\tau_g(Y_h)\otimes\tau_{gh}(Z))\xrightarrow{\iota\otimes\psi_{g,h}^{-1}} X\otimes\tau_g(Y\otimes\tau_h(Z))
$$
for $X\in\CC_g$ and $Y\in\CC_h$ ($\Phi$ is the associativity morphism in $\CC$) defines a monoidal category, which we denote by $\CC^\tau$.
\end{remark}

\section{Quantum subgroups of twisted algebraic groups}
\label{sec:twist-groups-acti}

In this section we extend some of our results on graded twistings of compact groups~\cite{MR3580173} to affine algebraic groups.

\subsection{Quotients of twisted Hopf algebras}\label{ssec:quotients}

We start with an arbitrary Hopf algebra $A$ endowed with an invariant cocentral action $(p, \alpha)$ of $\Gamma$.
Assume $I\subset A$ is a proper $\Gamma$-graded and $\Gamma$-stable Hopf ideal, so that $I = \bigoplus_g (I \cap A_g)$ and $\alpha_g(I)\subset I$ for all $g\in\Gamma$. Then $I\subset\ker\eps$, hence $I\subset\ker p=\bigoplus_g A_g^+$, so the invariant cocentral action $(p, \alpha)$ induces an invariant cocentral action~$(\overline{p}, \overline{\alpha})$ on $A/I$. We can then form the graded twisting $(A/I)^{t, \overline{\alpha}}$, which is nothing else than the quotient $A^{t, \alpha}/j(I)$, with $j$ as in~\eqref{eq:def-of-j}.

Recall that even when $\alpha_g\otimes\iota$ are not algebra automorphisms of $A^{t,\alpha}$, a subspace $X \subset A^{t, \alpha}$ is said to be $\Gamma$-stable if $(\alpha_g\otimes \iota)(X)\subset X$ for all $g \in \Gamma$. The following is a reformulation of \cite{MR3580173}*{Proposition~4.2}.

\begin{proposition}
\label{prop:correspideal}
Let $A$ be a Hopf algebra endowed with an invariant cocentral action $(p, \alpha)$ of a group~$\Gamma$. Let $\IdSet^{p,\alpha}(A)$ (resp. $\IdSet^{\tilde{p},\tilde{\alpha}}(A^{t,\alpha})$) denote the set of $\Gamma$-stable and $\Gamma$-graded ideals of $A$ (resp. of~$A^{t,\alpha}$). Then we have a bijection between these sets given by
$$
\IdSet^{p,\alpha}(A) \ni I = \bigoplus_{g \in \Gamma} I_g \mapsto j(I)=\bigoplus_{g \in \Gamma} I_g\otimes g \in \IdSet^{\tilde{p},\tilde{\alpha}}(A^{t,\alpha}).
$$
This bijection respects the Hopf ideals on both sides, and if $I\subset A$ is a proper $\Gamma$-graded and $\Gamma$-stable Hopf ideal, then $A^{t,\alpha}/j(I) \simeq (A/I)^{t, \overline{\alpha}}$.
\end{proposition}

It is worth mentioning that this result can also be easily established using well-known properties of strongly graded rings. Namely, by~\cite{MR676974}*{Corollary~I.3.2.9} the map $I\mapsto I\cap A_e$ defines a bijection between the graded ideals of $A$ and the ideals $J$ of $A_e$ such that $A_gJA_{g^{-1}}\subset J$ for all $g\in\Gamma$. The same is true for $A^{t,\alpha}$. Then the first statement of the above proposition follows from the observation that a $\Gamma$-stable ideal $J\subset A_e$ has the property $A_gJA_{g^{-1}}\subset J$ if and only if the $\Gamma$-stable ideal $j(J)=J\otimes 1$ of $(A^{t,\alpha})_e\simeq A_e$ has the property $j(A_g)j(J)j(A_{g^{-1}})\subset j(J)$.

\smallskip

In view of Proposition~\ref{prop:correspideal}, a natural question is under which conditions a Hopf ideal of $A$ or~$A^{t,\alpha}$ is $\Gamma$-stable and $\Gamma$-graded. The next result shows that under quite general assumptions the only reason why a Hopf ideal may be nongraded is the existence of nontrivial quotients of~$\Gamma$.

\begin{proposition}\label{prop:cocentralexact}
Let $p\colon A \rightarrow K\Gamma$ be a surjective cocentral Hopf algebra homomorphism and $f\colon A\rightarrow L$ be a surjective Hopf algebra homomorphism. If $L$ is faithfully flat over $f(A_e)$ as a left or right module, then there exists a surjective group morphism $u\colon\Gamma \rightarrow \bar{\Gamma}$ and a commutative diagram of Hopf algebra homomorphisms
$$
\xymatrix{
K \ar[r] & A_e \ar[r]^i \ar[d]^{f|_{A_e}} & A \ar[r]^p \ar[d]^f & K \Gamma \ar[r] \ar[d]^u & K\\
K \ar[r] & f(A_e) \ar[r] &  L \ar[r]^q & K \bar{\Gamma} \ar[r] & K
}
$$
with exact (as sequences of Hopf algebras) rows, where $q$ is cocentral.

The faithful flatness assumption is satisfied in each of the following cases:
\begin{enumerate}
\item\label{it:A-e-comm}  $A_e$ is commutative;
\item\label{it:L-coss}  $L$ is cosemisimple;
\item\label{it:L-ptd}  $L$ is pointed, that is, simple comodules are one-dimensional.
\end{enumerate}
\end{proposition}

\begin{proof}
Since $A_e^+A=AA_e^+$ by Proposition~\ref{pbasic}, we have $f(A_e)^+L=Lf(A_e)^+$. Hence we can form the quotient Hopf algebra $Q = L/f(A_e)^+L$, together with the canonical surjection $q\colon L\to Q$. Since $L$ is faithfully flat over $f(A_e)$, by~\cite{MR1243637}*{Proposition~3.4.3} we get an exact sequence
$$
K \to f(A_e) \to L \xrightarrow{q}Q\to K
$$
of Hopf algebra homomorphisms. Since we also have $\ker p= A_e^+A$, we see that $qf$ vanishes on $\ker p$ and there exists a surjective Hopf algebra homomorphism $u\colon K\Gamma \rightarrow Q$ such that $u p = qf$. It follows that we can identify $Q$ with $K\bar\Gamma$ for a quotient $\bar\Gamma$ of $\Gamma$ in such a way that $u\colon K\Gamma\to K\bar\Gamma$ is induced by the factor map $\Gamma\to\bar\Gamma$. The cocentrality of $q$ follows then from that of $p$. This finishes the proof of the first part of the proposition.

As for the concrete conditions that ensure faithful flatness of $L$ over $f(A_e)$, for the case~\eqref{it:A-e-comm} this is~\cite{MR1954457}*{Proposition 3.12}, for~\eqref{it:L-coss} this is~\cite{MR3263140}, while for~\eqref{it:L-ptd} see, e.g.,~\citelist{\cite{MR0321963}\cite{MR0437582}}.
\end{proof}

\begin{corollary}\label{cor:cocentralexact}
Under the assumptions of the proposition, if $\Gamma$ is simple, then either $L=f(A_e)$ or $\ker f\subset A$ is a proper $\Gamma$-graded Hopf ideal.
\end{corollary}

\bp By assumption, either $\bar\Gamma$ is trivial or $u$ is an isomorphism. In the first case we have $f(A_e)=L^{\co K\bar\Gamma}=L$. In the second case $\ker f\subset\ker p$ and hence $\ker f$ is $\Gamma$-graded.
\ep

\begin{remark}
It is not true that a Hopf algebra is always faithfully flat over its Hopf subalgebras~\cite{MR1761130}. It is, however, still open whether this is true for Hopf algebras with bijective antipode.
\end{remark}

\subsection{Quotients of twisted commutative Hopf algebras}

We now turn to graded twistings of commutative Hopf algebras. In that case there is no loss of generality in assuming that $\Gamma$ is abelian.

\begin{lemma}\label{lem:correspidealcommutative}
Let $A$ be a commutative Hopf algebra endowed with an invariant cocentral action~$(p, \alpha)$ of a group~$\Gamma$. If $J$ is a $\Gamma$-graded ideal of $A^{t, \alpha}$, then $J$ is automatically $\Gamma$-stable.
\end{lemma}

\begin{proof}
Since $J$ is $\Gamma$-graded, we can write $J =\bigoplus_{g \in \Gamma} J_g \otimes g$ with $J_g \subset A_g$. Then
$$
J (A_h \otimes h) = \bigoplus_{g \in \Gamma} J_g\alpha_g(A_h) \otimes gh\subset J,
$$
hence $J_g A_h = J_g \alpha_g(A_h) \subset J_{gh}$ for all $g,h$. Similarly, from
$$
(A_h \otimes h)J = \bigoplus_{g \in \Gamma} A_h\alpha_h(J_g) \otimes hg\subset J
$$
we obtain $A_h\alpha_h(J_g) \subset J_{hg}$ for all $g,h$. We thus have
$$
\alpha_g(J_h)\subset\alpha_g(A_eJ_h)=A_e\alpha_g(J_h)= A_{g^{-1}}A_g\alpha_g(J_h) \subset A_{g^{-1}}J_{gh}.
$$
But by the commutativity of $A$ the last term is equal to $J_{gh}A_{g^{-1}}\subset J_h$. This shows the $\Gamma$-stability of $J$.
\end{proof}

Combining this with Corollary~\ref{cor:cocentralexact} we get the following result.

\begin{theorem}\label{thm:quotcomprime}
Suppose that $\Gamma$ is a cyclic group of prime order and $A$ is a commutative Hopf algebra endowed with an invariant cocentral action $(p, \alpha)$ of $\Gamma$ such that $p\colon A\to K\Gamma$ is surjective. Then for any Hopf algebra quotient $L$ of $A^{t,\alpha}$ one of the following holds:
\begin{enumerate}
\item\label{it:concl-1} $L$ is commutative and isomorphic to a Hopf algebra quotient of $A_e$;
\item\label{it:concl-2} $L$ is isomorphic to $(A/I)^{t,\overline{\alpha}}$, for a $\Gamma$-graded and $\Gamma$-stable Hopf ideal $I \subset A$.
\end{enumerate}
\end{theorem}

\begin{proof}
Write $L$ as $A^{t, \alpha}/J$. Since $(A^{t,\alpha})_e\simeq A_e$ is commutative, we can apply Corollary~\ref{cor:cocentralexact} to~$A^{t, \alpha}$ and conclude that either $L$ is a quotient of $(A^{t,\alpha})_e$ or $J$ is $\Gamma$-graded. In the second case, by the previous lemma, $J$ is also $\Gamma$-stable and then, by Proposition~\ref{prop:correspideal}, we have $J=j(I)$ for a $\Gamma$-graded and $\Gamma$-stable Hopf ideal $I \subset A$, so that $L\simeq(A/I)^{t,\overline{\alpha}}$.
\end{proof}

Recall that by Cartier's theorem $A$ has no nilpotent elements~\citelist{\cite{MR0148665}\cite{MR547117}*{Section 11.4}}. Let us further assume that $K$ is algebraically closed and $A$ is finitely generated over $K$, so that~$A$ is isomorphic to the algebra $\Pol(G)$ of regular functions on some affine algebraic group~$G$ defined over $K$ (which can be identified with a Zariski closed subgroup of $\GL_n(K)$ for some~$n$). In this case an invariant cocentral action of $\Gamma$, with surjective homomorphism $p\colon A\to K\Gamma$, amounts to a pair $(i, \alpha)$, where
\begin{itemize}
 \item $i\colon\hat{\Gamma}={\rm Hom}(\Gamma, K^\times) \rightarrow Z(G)$ is an injective morphism of algebraic groups,
\item $\alpha\colon \Gamma \rightarrow \Aut(G)$ is a group morphism,
\end{itemize}
such that for all $g \in \Gamma$ and $\psi \in \hat{\Gamma}$, $\alpha_g(i(\psi))=i(\psi)$. By an invariant cocentral action of $\Gamma$ on~$G$ we will mean such a pair $(i,\alpha)$ and denote by $\Pol(G^{t, \alpha})$ the Hopf algebra $\Pol(G)^{t, \alpha}$. See Subsection \ref{ssec:q-plane} for concrete examples of this kind, or more generally \cite{MR3580173}.

Note that the $\Gamma$-grading on $\Pol(G)$ is described as
$$
\Pol(G)_g=\{a \in \Pol(G) \ | \ a(i(\psi)x)=\psi(g)a(x) \ \forall x \in G, \ \forall \psi\in \hat{\Gamma} \} \quad (g \in \Gamma).
$$
In particular, $\Pol(G)_e=\Pol(G/i(\hat\Gamma))$.
In this setting, a proper $\Gamma$-graded Hopf ideal of $\Pol(G)$ is defined by an algebraic subgroup $H\subset G$ such that $i(\hat{\Gamma})\subset H$. Such an ideal is $\Gamma$-stable if and only if $H$ is globally invariant under the automorphisms $\alpha_g$, $g\in\Gamma$, in which case we say that~$H$ itself is $\Gamma$-stable.

In this language, Theorem~\ref{thm:quotcomprime} says that if $\Gamma$ is a cyclic group of prime order, then any Hopf algebra quotient of $\Pol(G^{t, \alpha})$ is isomorphic either to $\Pol(H)$ for an algebraic subgroup $H\subset G/i(\hat\Gamma)$ or to $\Pol(H^{t,\alpha})$ for a $\Gamma$-stable algebraic subgroup $H\subset G$  such that $i(\hat{\Gamma})\subset H$.

We next want to analyze when the algebra $\Pol(G^{t,\alpha})$ is noncommutative. For this, it is convenient, as in~\cite{MR3580173}, to consider the following subgroup of $G$:
$$
G_1=\{ x \in G \ | \ \Gamma. x \subset i(\hat{\Gamma})x\},
$$
where $\Gamma.x=\{\alpha_g(x)\}_{g\in\Gamma}$.

\begin{lemma}\label{lem:commtwi}
 Let $(i,\alpha)$ be an invariant cocentral action of a finite abelian group $\Gamma$ on an affine algebraic group $G$. Assume that $G_1\ne G$. Then the algebra $\Pol(G^{t, \alpha})$ is noncommutative. Conversely, if $\Pol(G^{t, \alpha})$ is noncommutative and $\Gamma$ is cyclic, then $G_1\ne G$.
\end{lemma}

\begin{proof}
First assume $G \setminus G_1 \neq \emptyset$. Take an element $x$ from this set and choose $g \in \Gamma$ such that $\alpha_g(x) \not \in i(\hat{\Gamma})x$. We then have
$$
i(\hat{\Gamma})\alpha_{g^{-1}}(x) \cap i(\hat{\Gamma})x = \emptyset.
$$
Hence, for every $h \in \Gamma$, we can find a function $a_h \in \Pol(G)$ such that it vanishes on $i(\hat{\Gamma})\alpha_{g^{-1}}(x)$ and
$$
a_h(i(\psi)x)=\psi(h)\ \ \text{for all}\ \psi\in\hat\Gamma.
$$
Define $a'_h\in\Pol(G)$ by
$$
a'_h(y)= \frac{1}{|\Gamma|}\sum_{\psi \in \hat{\Gamma}}\psi(h)^{-1}a_h(i(\psi)y).
$$
Then we have $a'_h \in \Pol(G)_h$, $a'_h(x)=1$, and $a'_h(\alpha_{g^{-1}}(x))=0$. In particular we have $a'_g \alpha_g(a'_e) \neq  a'_e a'_g$ by comparing their values at $x$, and
$$
(a'_g\otimes g) \cdot (a_e'\otimes 1) \neq  (a'_e\otimes 1) \cdot (a_g'\otimes g)
$$
in the crossed product, which shows that $\Pol(G^{t, \alpha})$ is indeed noncommutative.

Conversely, assume that $G_1=G$. Then for all $x \in G$ and $g \in \Gamma$, there exists $\psi_{x,g} \in \hat{\Gamma}$ such that $\alpha_g(x)=i(\psi_{x,g})x$. If we fix $x \in G$, we obtain a bicharacter
$$
\Gamma \times \Gamma \rightarrow K^\times, \quad (g,h) \mapsto \psi_{x,g}(h).
$$
If we further assume that $\Gamma$ is cyclic, this must be a symmetric bicharacter. It follows that for all $a \in \Pol(G)_g$ and $b \in \Pol(G)_h$, we have $a\alpha_g(b) = b\alpha_h(a)$, or in other words
$$
(a \otimes g) \cdot (b \otimes h) = (b \otimes h) \cdot (a \otimes g),
$$
which shows that $\Pol(G^{t, \alpha})$ is commutative.
\end{proof}

This lemma together with Theorem~\ref{thm:quotcomprime} give the following.

\begin{theorem}\label{thm:quotcomprime2}
Let $\Gamma$ be a cyclic group of prime order, $(i, \alpha)$ be an invariant cocentral action of~$\Gamma$ on an affine algebraic group $G$ defined over an algebraically closed field $K$ of characteristic zero, and $L$ be a Hopf algebra quotient of $\Pol(G^{t,\alpha})$.
Then one of the following assertions holds:
\begin{enumerate}
\item $L$ is commutative and is isomorphic to $\Pol(H)$ for an algebraic subgroup $H \subset G/i(\hat{\Gamma})$;
\item $L$ is commutative and is isomorphic to  $\Pol(H^{t,\alpha})$ for a  $\Gamma$-stable algebraic subgroup $H\subset G$ such that $i(\hat{\Gamma})\subset H$ and  $\{x \in H \ | \ \Gamma. x \subset i(\hat{\Gamma})x\}= H$;
\item $L$ is noncommutative and is isomorphic to  $\Pol(H^{t,\alpha})$, for a  $\Gamma$-stable algebraic subgroup $H\subset G$ such that $i(\hat{\Gamma})\subset H$ and $\{x \in H \ | \ \Gamma. x \subset i(\hat{\Gamma})x\}\ne H$.
\end{enumerate}
\end{theorem}

\begin{remark}
As opposed to cases (ii) and (iii), we do not claim that (i) is actually realized, that is, given an algebraic subgroup $H \subset G/i(\hat{\Gamma})$ it is not always true that $\Pol(H)$ can be obtained as a quotient of $\Pol(G^{t,\alpha})$.
\end{remark}

For complexifications of compact Lie groups a similar result describing Hopf $*$-algebra quotients of $\Pol(G^{t,\alpha})$ was obtained in~\cite{MR3580173}*{Theorem 4.8} using C$^*$-algebraic tools.

\section{Twisting of comodule algebras}
\label{sec:twist-comod-algebr}

\subsection{Comodule algebras and their coinvariants}\label{ssec:twist-comod}

As before, let $A$ be a Hopf algebra and~$(p,\alpha)$ be an invariant cocentral action of $\Gamma$ on $A$. Recall that by \eqref{eq:comod-eqv-for-twist} we have an equivalence between the categories $\Comod(A)$ and $\Comod(A^{t,\alpha})$. We now want to consider more refined classes of comodules and comodule algebras respecting the action of $\Gamma$.

\begin{definition}
Let $\rho_V\colon V \to A \otimes V$ be a left $A$-comodule, and $\beta_V=\beta$ be a linear representation of $\Gamma$ on $V$. We say that $(V, \beta)$ is \emph{$\Gamma$-equivariant}, or that it is a \emph{$\Gamma$-$A$-comodule}, if $\rho$ is $\Gamma$-equivariant with respect to $\beta$ and $\alpha \otimes \beta$. We denote the category of $\Gamma$-$A$-comodules with the $\Gamma$-linear and $A$-colinear maps as morphisms by $\Comod(A,\Gamma,  \alpha)$.
\end{definition}

As a basic example of a $\Gamma$-$A$-comodule, we have of course $A=(A,\alpha)$ itself.

\begin{remark}
The action $\alpha$ endows $A$ with a left $K\Gamma$-module coalgebra structure, and the category just defined is the category of Doi--Hopf modules $\tensor*[^A_{K\Gamma}]{{\textbf M}(K\Gamma)}{}$~\cite{MR1198206}. The aim of our present terminology is to emphasize the $A$-comodule structure. Note that in \cite{MR1386496}, the authors focus on the category of relative Hopf modules $\tensor*[^{K\Gamma}_A]{\mathcal M}{}$ instead.
\end{remark}

It is straightforward to check that the monoidal structures on ${\rm Comod}(A)$ and ${\rm Mod}(K\Gamma)$ induce a monoidal structure on $\Comod(A,\Gamma,  \alpha)$, with the  obvious forgetful functors being strict monoidal.

\begin{definition}
A \emph{$\Gamma$-$A$-comodule algebra} is given by the following data:
\begin{itemize}
\item $(R, \rho_R \colon R \to A \otimes R)$ is a left $A$-comodule algebra, and
\item $\beta\colon \Gamma \rightarrow \Aut(R)$ is an action of $\Gamma$ on $R$ by algebra automorphisms,
\end{itemize}
such that $\rho_R \beta_g= (\alpha_g \otimes \beta_g)\rho_R$ for every $g \in \Gamma$.
\end{definition}

In other words, a $\Gamma$-$A$-comodule algebra is just an algebra object in the monoidal category $\Comod(A,\Gamma,\alpha)$.

By analogy with the construction of $A^{t,\alpha}$ we can define twistings of $\Gamma$-$A$-comodule algebras. Let $R=(R, \beta)$ be such an algebra. Recall that the $A$-comodule structure map composed with $p \otimes \iota$ defines a $K\Gamma$-comodule algebra structure on $R$, and we have a direct sum decomposition $R = \bigoplus_{g\in\Gamma} R_g$. Note that by the equivariance condition we have $\beta_g(R_h)=R_h$ for all $g,h \in \Gamma$. Now, identifying $R\otimes K\Gamma$ with $R\rtimes_\beta\Gamma$ as a linear space, we obtain an algebra structure on
$$
R^{t,\alpha}=\bigoplus_{g\in\Gamma}R_g\otimes g\subset R\rtimes_\beta\Gamma
$$
defined by that on the crossed product. We denote the $A^{t,\alpha}$-comodule $R^{t,\alpha}$ with this algebra structure by $R^{t,\alpha,\beta}$, or simply by $R^{t,\beta}$.  Similarly to Remark \ref{rem:previoustwist}, the algebra structure of $R^{t,\alpha,\beta}$ is a special instance of the construction from \cite{MR1367080}.

The following property is immediate by definition.

\begin{proposition}\label{prop:twist_coinv}
For any $\Gamma$-$A$-comodule algebra $R$, we have a canonical isomorphism $R_e \simeq (R^{t, \beta})_e$ of algebras. Furthermore, if we identify the Hopf algebras $A_e$ and $(A^{t,\alpha})_e$, then this is an isomorphism of $A_e$-comodule algebras. In particular, we have $^{\co A^{t,\alpha}}\!(R^{t, \beta}) \simeq {^{\co A}\!R}$.
\end{proposition}

\begin{remark}\label{rem:twist_coinv}
It will be useful to slightly strengthen this simple observation. Assume we are given a proper $\Gamma$-graded and $\Gamma$-stable Hopf ideal $I\subset A$ as in the previous section. Then a $\Gamma$-$A$-comodule algebra $R$ can also be viewed as a $\Gamma$-$(A/I)$-comodule algebra. The construction of~$R^{t,\beta}$ is independent of the point of view. Hence by applying the above proposition to $A/I$ instead of $A$, we get $^{\co (A/I)^{t,\bar\alpha}}\!(R^{t, \beta}) \simeq {^{\co (A/I)}\!R}$.
\end{remark}

Let assume now that $\Gamma$ is abelian. Then $A^{t,\alpha}$ is equipped with the action of $\Gamma$ by the algebra automorphisms $\tilde\alpha_g=\alpha_g\otimes\iota$ and the monoidal category $\Comod(A^{t,\alpha},\Gamma,\tilde\alpha)$ is well-defined. Similarly, if $(R,\beta)$ is a  $\Gamma$-$A$-comodule algebra, then we have an action of $\Gamma$ on $R^{t,\beta}$ by the automorphisms $\tilde\beta_g=\beta_g\otimes\iota$, so that $(R^{t,\beta},\tilde\beta)$ becomes a $\Gamma$-$A^{t,\alpha}$-comodule algebra.

The functor $F\colon\Comod(A)\to\Comod(A^{t,\alpha})$ extends in the obvious way to a functor $$\Comod(A,\Gamma,  \alpha) \to \Comod(A^{t, \alpha}, \Gamma,  \tilde{\alpha}),$$ which we continue to denote by $F$. It is easy to see that this is an equivalence of categories, but in fact the following stronger result is true.

\begin{theorem}\label{thm:equivrelat}
If $\Gamma$ is abelian, then the functor $F\colon \Comod(A,\Gamma,  \alpha) \to \Comod(A^{t, \alpha}, \Gamma,  \tilde{\alpha})$
can be enriched to an  equivalence of monoidal categories.
\end{theorem}

\begin{proof}
We have to specify a natural family of isomorphisms
$$
F_2(V,W)\colon V^{t,\alpha} \otimes W^{t, \alpha} \rightarrow (V \otimes W)^{t, \alpha}
$$
in the category $\Comod(A^{t, \alpha},\Gamma,  \widetilde{\alpha})$ for $V=(V,\beta)$ and $W =(W, \gamma)$ in $\Comod(A,\Gamma, \alpha)$, satisfying the compatibility conditions with unit and associator. We claim that it can be given as
\begin{equation}\label{eq:tilde-F-def}
F_2(V,W)((v \otimes g) \otimes (w \otimes h)) = v \otimes \gamma_g(w) \otimes gh \quad (v \in V_g, w \in W_h).
\end{equation}
Indeed, for $v \in V_g$ and $w \in W_h$, we have $v\otimes \gamma_g(w) \in (V \otimes W)_{gh}$, hence our map is well-defined. It is a linear bijection, since
$$
(V \otimes W)^{t, \alpha}= \bigoplus_{g \in \Gamma} (V \otimes W)_g \otimes g =
\bigoplus_{\substack{g \in \Gamma\\ rs=g}} V_r \otimes W_s \otimes g = \bigoplus_{r,s \in \Gamma} V_r \otimes W_s \otimes rs
$$
and our map identifies $V_g \otimes g \otimes V_h \otimes h$ and  $V_g \otimes V_h \otimes gh$. Next let us verify the $\Gamma$-linearity. For $r \in \Gamma$, we have
\begin{multline*}
F_2(V,W)(r\cdot(v \otimes g \otimes w \otimes h))=F_2(V,W)(\beta_r(v) \otimes g \otimes \gamma_r(w)\otimes h)=\beta_r(v) \otimes  \gamma_{gr}(w)\otimes gh \\
= \beta_r(v) \otimes \gamma_{rg}(w)\otimes gh = r\cdot (v \otimes  \gamma_{g}(w)\otimes gh) =  r\cdot F_2(V,W)(v \otimes g \otimes w \otimes h),
\end{multline*}
hence $F_2(V,W)$ is $\Gamma$-linear. The $A^{t,\alpha}$-colinearity comes from the identities
\begin{align*}
 v_{(-1)}\gamma_g(w)_{(-1)} \otimes gh \otimes  & v_{(0)} \otimes \gamma_g(w)_{(0)} \otimes gh  =
 v_{(-1)}\alpha_g(w_{(-1)})\otimes  gh \otimes v_{(0)} \otimes \gamma_g(w_{(0)}) \otimes gh  \\
&= \left((v_{(-1)}\otimes g)(w_{(-1)}\otimes h)\right) \otimes F_2(V,W)(v_{(0)} \otimes g  \otimes w_{(0)} \otimes h).
\end{align*}

We now know that $F_2(V,W)$ is a natural transformation $F_2\colon F(V) \otimes F(W) \to F(V \otimes W)$. Since the monoidal unit is represented by $K$ concentrated at degree $e$, $F_2$ is compatible with the unit. As for the compatibility with associator, since both $\Comod(A)$ and $\Comod(A^{t, \alpha})$ are strict it amounts to verifying the identity
$$
F_2({F(V\otimes W),Z}) (F_2(V,W) \otimes \iota) = F_2({V, F(W\otimes Z)}) (\iota \otimes F_2({W,Z}))
$$
for $(V, \beta)$, $(W, \gamma)$, and $(Z, \theta)$ in $\Comod(A,\Gamma,\alpha)$. It can be directly checked that both sides are characterized by
$$
(v \otimes g) \otimes (w \otimes h) \otimes (z \otimes k) \mapsto v \otimes \gamma_g(w) \otimes \theta_{g h}(z) \otimes g h k
$$
for $v \in V_g$, $w \in W_h$, and $z \in Z_k$. This completes the proof.
\end{proof}

Consequently, $F$ induces an equivalence between the respective categories of algebra objects, that is, the categories of $\Gamma$-$A$-comodule algebras and of $\Gamma$-$A^{t,\alpha}$-comodule algebras. It is easy to check that this recovers the construction of $R^{t,\beta}$ from $(R,\beta)$ introduced above for arbitrary $\Gamma$.

\begin{remark}
A similar result can be proved in the $\tau$-setting discussed in Remark~\ref{rem:previoustwist}; in particular, a form of the above theorem is true for non-abelian groups. Namely, let $A$ be a $\Gamma$-graded Hopf algebra equipped with coalgebra automorphisms $\tau_g$ as in that remark. Denote by $\Comod_\tau(A,\Gamma)$ the category of triples $(V,\rho,\beta)$, where $(V,\rho)$ is a  left $A$-comodule (inheriting in this way a $\Gamma$-grading) and $\beta$ is a (not necessarily multiplicative) map from~$\Gamma$ into the group of linear automorphisms of $V$ such that $\rho\beta_g=(\tau_g\otimes\beta_g)\rho$, so that the automorphisms $\beta_g$ are graded automorphisms of $V$. Define a tensor product on this category by taking the usual tensor product of  comodules~$V$ and~$W$ and equipping it with the maps
$$
(\beta\otimes \gamma)_g(v\otimes w)=\beta_g(v)\otimes\gamma_{gh}\gamma_h^{-1}(w), \ v\in V_h,\ w\in W.
$$
This makes $\Comod_\tau(A,\Gamma)$ into a strict monoidal category.

Consider now the $\tau$-twisting $A^\tau$ and equip it with the maps $\tilde\tau_g=\tau_g^{-1}$, so that $(A^\tau)^{\tilde\tau}=A$. Any $A$-comodule $V$ can be considered as a $A^\tau$-comodule $V^\tau$, and if $V$ is equipped with linear automorphisms $\beta_g$ as above, then we equip $V^\tau$ with the linear automorphisms $\tilde\beta_g=\beta_g^{-1}$. We thus get a functor $F\colon\Comod_\tau(A,\Gamma)\to\Comod_{\tilde\tau}(A^\tau,\Gamma)$, $F(V)=V^\tau$. Similarly to the above theorem, it can then be easily checked that $F$, being equipped with the tensor structure
$$
F_2(V,W)\colon V^\tau\otimes W^\tau\to (V\otimes W)^\tau,\ \ v\otimes w\mapsto v\otimes\gamma_h(w),\ \ v\in V_h,\ w\in W,
$$
becomes an equivalence of monoidal categories.
\end{remark}

\subsection{Quantum planes}\label{ssec:q-plane}

As an example, let us explain how the framework developed above can be used to obtain information on the fixed point algebras for quantum groups actions on quantum planes. Throughout this section we assume that $K$ is algebraically closed and, as usual, has characteristic zero.

Let us recall the two-parameter quantum group $\GL_{p,q}(2)$ for $p,q \in K^\times$ which can be found in, e.g., \cite{MR1432363}. Its coordinate algebra $\Pol(\GL_{p,q}(2))$ is the algebra presented by generators
$a$, $b$, $c$, $d$, $\delta^{-1}$ subject to the relations
\begin{gather*}
ba=qab, \ dc=qcd, \ ca =pac, \ db=pbd, \ qcb=pbc,\\
da-ad= (p-q^{-1})bc, \ (ad-q^{-1}bc)\delta^{-1}=1=\delta^{-1}(ad-q^{-1}bc).
\end{gather*}
The  Hopf algebra structure on $\Pol(\GL_{p,q}(2))$ is given by the usual formulas
\begin{gather*}
\Delta(a)=a\otimes a +b \otimes c, \ \Delta(b)=a \otimes b +b\otimes d, \ \Delta(c)=c\otimes a+d\otimes c, \
\Delta(d)=c \otimes b + d\otimes d,\\
\Delta(\delta^{-1})=\delta^{-1}\otimes \delta^{-1}, \ \varepsilon(a)=\varepsilon(d)=\varepsilon(\delta^{-1}) = 1, \ \varepsilon(b)=\varepsilon(c)=0,\\
S(a)=d\delta^{-1}, \ S(b)=-qb\delta^{-1}, \ S(c)=-q^{-1}c\delta^{-1}, \ S(d)=a\delta^{-1}, \ S(\delta^{-1})=ad-q^{-1}bc.
\end{gather*}
Let us denote the quantum plane algebra $K\langle x, y \mid y x = p x y \rangle$ by $K_p [x,y]$. There is an algebra homomorphism $\rho\colon K_p [x,y] \rightarrow \Pol(\GL_{p,q}(2)) \otimes K_p[x,y]$ characterized by
\begin{equation*}
\quad \rho(x) = a \otimes x + b \otimes y, \quad \rho(y) = c \otimes x + d \otimes y
\end{equation*}
This defines a left $\Pol(\GL_{p,q}(2))$-comodule algebra structure on $K_p[x,y]$.

Denote by $\Gamma_q$ the cyclic group of the same order as $q$ as an element in the multiplicative group $K^\times$, and fix its generator $g$. There is a cocentral Hopf algebra homomorphism
\begin{equation*}
 p\colon \Pol(\GL_{p,q}(2)) \rightarrow K\Gamma_q, \quad
\begin{pmatrix}
 a & b \\
c & d
\end{pmatrix}
\mapsto
\begin{pmatrix}
 g & 0 \\
0 & g
\end{pmatrix}.
\end{equation*}
The dual group $\hat{\Gamma}_q$ can be identified with a subgroup of $K^\times$ by taking the image of $g$: it becomes $\mu_N(K)$ if $q$ has finite order $N\geq 1$, otherwise it is $K^\times$. The above Hopf algebra homomorphism gives an identification of $\hat{\Gamma}_q$ as a quantum subgroup of $\GL_{p,q}(2)$. For a quantum subgroup $G \subset \GL_{p,q}(2)$, we write $\hat{\Gamma}_q \subset G$ if the corresponding surjective Hopf algebra homomorphism $\Pol(\GL_{p,q}(2)) \rightarrow \Pol(G)$ factorizes $p$.

As shown in~\cite{MR3580173}*{Example 4.11}, $\GL_{q^{-1},q}(2)$ is a graded twisting of $\GL(2)$, and more generally $\GL_{q^{-1}\xi,q\xi^{-1}}(2)$ is a graded twisting of $\GL_{q^{-1},q}(2)$, for every $\xi \in K^\times$. Denote by $\alpha_g$  the Hopf algebra automorphism of $\Pol(\GL(2))$ defined by
$$\alpha_g\left(\begin{pmatrix}
 a & b \\
c & d
\end{pmatrix}\right)= \begin{pmatrix}
 a & q^{-1}b \\
qc & d
\end{pmatrix}$$
and consider the group morphism $\alpha\colon \Gamma_q \rightarrow \Aut(\Pol(\GL(2))$, $g \mapsto \alpha_g$. We then have that $(p,\alpha)$, or $(i, \alpha)$ with the inclusion map $i \colon \hat{\Gamma}_q \to Z(\GL(2))$ as scalar matrices, is an invariant cocentral action of $\Gamma_q$ on $\Pol(\GL(2))$. This leads to the identification
\begin{equation*}
\Pol(\GL_{q^{-1},q}(2)) \simeq \Pol(\GL(2))^{t, \alpha}, \quad
\begin{pmatrix}
 a & b \\
 c & d
\end{pmatrix} \mapsto \begin{pmatrix}
 a\otimes g & b \otimes g\\
 c\otimes g & d\otimes g
\end{pmatrix}.
\end{equation*}

By Proposition~\ref{prop:correspideal}, the Hopf ideals in $\IdSet^{p,\alpha}(\Pol(\GL(2))$ correspond to certain subgroups of~$\GL(2)$. Let us give them a name.

\begin{definition}
We say that an algebraic subgroup $H \subset \GL(2)$ is \emph{$q$-admissible} if $\hat{\Gamma}_q \subset H$ and $\alpha_g(H)=H$.
\end{definition}

\begin{theorem}\label{thm:gl2}
 Let $G \subset \GL_{q^{-1},q}(2)$ be a quantum subgroup such that $\hat{\Gamma}_q \subset G$.
Then there exists a $q$-admissible algebraic subgroup
$G'\subset \GL(2)$  such that
$$K_{q^{-1}}[x,y]^G \simeq K[x,y]^{G'}$$
as algebras.
\end{theorem}

\begin{proof}
Let $G \subset \GL_{q^{-1},q}(2)$ be a quantum subgroup such that $\hat{\Gamma}_q \subset G$, and consider the corresponding surjective Hopf algebra homomorphism $f\colon \Pol(\GL_{q^{-1},q}(2)) \rightarrow \Pol(G)$. The assumption $\hat{\Gamma}_q \subset G$ precisely means that $\ker f$ is $\Gamma_q$-graded, hence, by Lemma~\ref{lem:correspidealcommutative}, $\ker f$ is also $\Gamma_q$-stable. By Proposition~\ref{prop:correspideal} there exists a $q$-admissible algebraic subgroup $H\subset \GL(2)$ such that $\Pol(G) \simeq \Pol(H^{t,\alpha})$.

Now, if we let $G'=H$, in order to prove the theorem, by Proposition~\ref{prop:twist_coinv} and Remark~\ref{rem:twist_coinv} it suffices to show that there exists an action $\beta$ of $\Gamma_q$ on $K[x,y]$ turning $K[x,y]$ into a $\Gamma_q$-$\Pol(\GL(2)$-comodule algebra such that
$$
K[x,y]^{t,\beta}\simeq K_{q^{-1}}[x,y]
$$
as $\Pol(\GL_{q^{-1},q}(2))$-comodule algebras. We define the action $\beta$ by $$\beta_g(x)=x \ {\rm and} \ \beta_g(y)=q y.$$ It is then a simple matter to verify that the map
\begin{equation*}
 K_{q^{-1}}[x,y] \rightarrow K[x,y]^{t, \beta}, \quad x \mapsto x \otimes g, \quad y \mapsto y\otimes g,
\end{equation*}
defines the required isomorphism.
\end{proof}

\begin{remark}
The theorem is of little interest (and trivial) if $q$ is not a root of unity. Indeed in that case we have $K^\times \subset G$, and $K_{q^{-1}}[x,y]^G\subset K_{q^{-1}}[x,y]^{K^\times}=K$.
\end{remark}

It is well-known that $\Pol(\GL_{q^{-1},q}(2))$ is a $2$-cocycle twisting of $\Pol(\GL(2))$ for a $2$-cocycle induced from the subgroup $T_q$ of diagonal matrices in $\GL(2)$ with entries in $\hat{\Gamma}_q$, see for example~\cite{MR1432363}. Hence a statement similar to Theorem \ref{thm:gl2} can be obtained from the usual transport result in the cocycle twisting case, for quantum subgroups $G$ containing $T_q$. Our formulation is more general, and as seen in the next corollary it includes $\SL_{-1}(2)$ and its quantum subgroups, which are not $2$-cocycle twistings of ordinary groups.

\begin{corollary}\label{cor:invSL-1(2)}
 Let $G \subset \SL_{-1}(2)$ be a  quantum subgroup. Assume that $G$ is nonclassical or, more generally, that  $\{\pm 1\}\subset G$.
 Then there exists a $(-1)$-admissible algebraic subgroup
$G'\subset \SL(2)$  such that
$$K_{-1}[x,y]^G \simeq K[x,y]^{G'}$$
as algebras. 
\end{corollary}

\begin{proof}
The fact that if $G$ is nonclassical then $\{\pm 1\}\subset G$ (and, moreover, $G=H^{t,\alpha}$ for a $(-1)$-admissible algebraic subgroup
$H\subset \SL(2)$) follows from Theorem~\ref{thm:quotcomprime2}. Therefore the result is a consequence of Theorem~\ref{thm:gl2}.
\end{proof}

Note that by Lemma \ref{lem:commtwi} a quantum group $G\subset \SL_{-1}(2)$ with $\{\pm 1\}\subset G$ is nonclassical if and only if the corresponding $G'\subset \SL(2)$ contains a matrix that is neither diagonal nor anti-diagonal.

\begin{remark}
A less precise version of Corollary \ref{cor:invSL-1(2)} can be obtained more directly as follows.  Let $G \subset \SL_{-1}(2)$ be a  quantum subgroup with $\{\pm 1\}\subset G$. We get a Hopf algebra  exact sequence
\begin{equation*}
K \to \mathcal O(H) \to \mathcal O(G) \to K\mathbb Z_2 \to
K\end{equation*}
where $H\subset {\rm PSL}(2)$ is an algebraic subgroup. Hence we have $K_{-1}[x,y]^G \simeq (K_{-1}[x,y]^{\mathbb Z_2})^H$. It is straightforward that $K_{-1}[x,y]^{\mathbb Z_2}=K[x^2,xy,y^2] = K[x,y]^{\mathbb Z_2}$, so we have
 $K_{-1}[x,y]^G \simeq (K[x,y]^{\mathbb Z_2})^H \simeq K[x,y]^{G'}$, where $G'=\pi^{-1}(H)$, with $\pi\colon \SL(2) \rightarrow {\rm PSL}(2)$ the canonical surjection. This direct reasoning, however, does not detect the $(-1)$-admissibility of $G'$.
\end{remark}

Our technique applies as well to actions on quantum Weyl algebras. Recall \cite{MR1176901} that the quantum
Weyl algebra $A_1(q^{-1})$ is the algebra presented by generators $x$, $y$ subject to the relation $yx-q^{-1}xy=1$.
The quantum group $\SL_{q^{-1}}(2)$ acts on $A_1(q^{-1})$, with the action given by  the same formula as for the action on $K_{q^{-1}}[x,y]$. The algebra $A_1(-1)$ is denoted $W_2$ in \cite{arXiv:1501.07881}, where the invariants under finite group actions are studied, see \cite{MR3119216} as well. We denote the ordinary Weyl algebra $A_1(1)$ simply by $A_1$. Similarly to the previous corollary, we have the following result, which in many cases reduces the invariant theory for a (quantum) group action on $W_2$ to the invariant theory for a group action of $A_1$ (see, e.g., \cite{MR1045737}).

\begin{theorem}
Let $G \subset \SL_{-1}(2)$ be a  quantum subgroup. Assume that $G$ is nonclassical or, more generally, that  $\{\pm 1\}\subset G$. Then there exists a $(-1)$-admissible algebraic subgroup
$G'\subset \SL(2)$  such that
$W_2^G \simeq A_1^{G'}$
as algebras.
\end{theorem}

\section{Graded twisting of module categories}
\label{sec:categorical-picture}

In this section we introduce twistings of equivariant module categories, which is a categorical counterpart of the construction in Section~\ref{sec:twist-comod-algebr}. Throughout this section we assume that all categories (except for $\underline{\Gamma}$ and $\underline{\Aut}(\DD)$), and correspondingly functors between them, are $K$-linear.

\subsection{Equivariant module categories}\label{ssec:twist-mod-cat}

Let $\DD$ be a category. By a weak action of~$\Gamma$ on~$\DD$ we mean a monoidal functor $\beta\colon\underline{\Gamma} \to \underline{\Aut}(\DD)$. In this case we also say that $\DD$ is a \emph{$\Gamma$-equivariant category}. In other words, a weak action is given by the following data:
 \begin{itemize}
 \item a family of autoequivalences $(\beta^g)_{g\in\Gamma}$ of $\DD$, with $\beta^e\simeq\Id_\DD$,
 \item natural isomorphisms $\eta_\beta^{g,h}\colon\beta^g \beta^h \to \beta^{g h}$,
 \end{itemize}
 satisfying the compatibility condition $\eta_\beta^{g, h k} \beta^g(\eta_\beta^{h,k}) = \eta_\beta^{g h, k} \eta_\beta^{g h}$. A weak action is called \emph{strict} if $\beta^e = \Id_\DD$ and $\eta_\beta^{g,h}$ are the identity transformations.

\smallskip

Let $\CC$ be a monoidal category. Recall~\citelist{\cite{MR1184424}\cite{MR1604452}\cite{MR1976459}} that a category $\DD$ is said to be a \emph{right $\CC$-module category} if we are given a monoidal functor $\CC^{\otimes\mathrm{op}}\to\underline{\End}(\DD)$, or, more concretely, a bifunctor $\DD \times \CC \to \DD, (X, U) \mapsto X \times U$ and natural isomorphisms
\begin{equation*}\label{eq:mod-cat-assoc}
\nu_X\colon X\times\un\to X\ \ \text{and}\ \ \mu^{U,V}_X\colon (X\times U)\times V \to X \times (U\otimes V)
\end{equation*}
satisfying standard compatibility conditions.
When $\CC$ is strict, we say that the right $\CC$-module category structure on $\DD$ is \emph{strict} if $\nu_X$ and $\mu^{U,V}_X$ are the identity morphisms. In general the morphisms $\nu$ are determined by $\mu$, so we will usually omit them.

In this terminology a weak action of $\Gamma$ on $\DD$ is the same thing as the structure of a right $\underline{\Gamma}$-module category on $\DD$: given a weak action $\beta \colon \Gamma \curvearrowright \DD$, the module category structure is defined by $X\times g=\beta^{g^{-1}}(X)$ and $\mu^{g,h}=\eta^{h^{-1},g^{-1}}_\beta$.

When $\DD$ and $\DD'$ are right $\CC$-module categories, a $\CC$-module functor from $\DD$ to $\DD'$ is given by a pair $(F, \theta)$, where $F$ is a functor $\DD \to \DD'$ and $\theta$ is a family of natural isomorphisms $F(X) \times U \to F(X \times U)$, compatible in the obvious way with the isomorphisms $\mu^{U,V}_X$ on both categories.

Applying this to the case $\CC=\underline{\Gamma}$ we get the notion of a $\Gamma$-equivariant functor between $\Gamma$-equivariant categories. This can be rephrased in terms of the data defining weak actions. Namely, let $\alpha\colon \Gamma \curvearrowright \DD$ and $\beta \colon \Gamma \curvearrowright \DD'$ be weak actions of $\Gamma$. Then a $\Gamma$-equivariant functor from~$\DD$ to~$\DD'$ is given by a functor $F \colon \DD \to \DD'$ and natural isomorphisms of functors $\zeta^g\colon F \alpha^g \to \beta^g F$ such that the diagrams
\begin{equation}\label{eq:eqv-ftr}
\xymatrix@C=4em{
 & F \alpha^h \alpha^k (X) \ar[rd]^{\zeta^h_{\alpha^k (X)}} \ar[ld]_{F(\eta_\alpha^{h,k})} & \\
F \alpha^{h k} (X) \ar[d]_{\zeta^{h k}_X} &  & \beta^h F \alpha^k (X) \ar[d]^{\beta^h(\zeta^k_X)} \\
 \beta^{h k} F (X) & & \beta^h \beta^k F (X) \ar[ll]^{\eta_\beta^{h,k}}
}
\end{equation}
commute for all $X \in \DD$. By abuse of notation we also write $F$ instead of $(F, (\zeta^g)_g)$. Such an equivariant functor is said to be \emph{strict} if $\zeta^g$ is the identity for all $g$. For monoidal $\Gamma$-equivariant categories (discussed already in Subsection~\ref{ssec:cat-twist}) and monoidal functors, we of course require $\zeta^g$ to be natural isomorphisms of monoidal functors.

\smallskip

Assume next we are given a monoidal category $\CC$ and a weak action $\alpha\colon \Gamma \curvearrowright \CC$.

\begin{definition}
By a \emph{right $\Gamma$-$\CC$-module category}, or a right $\Gamma$-equivariant $\CC$-module category, we mean a right $\CC$-module category $\DD$ equipped with a weak action $\beta\colon \Gamma \curvearrowright \DD$ and a family of natural isomorphisms $\beta_2^g\colon \beta^g(X)\times\alpha^g(U) \to \beta^g(X\times U)$ compatible with $\mu$, $\eta_\alpha$ and $\eta_\beta$ in the sense that the following two diagrams commute:
$$
\xymatrix@C=3em{
(\beta^g(X)\times\alpha^g(U))\times\alpha^g(V)\ar[r]^{\ \ \ \ \beta^g_2\times\iota} \ar[d]_{\mu^{\alpha^g(U),\alpha^g(V)}_{\beta^g(X)}} & \beta^g(X\times U)\times\alpha^g(V) \ar[r]^{\beta^g_2} & \beta^g((X\times U) \times V) \ar[d]^{\beta^g(\mu^{U,V}_X)}\\
\beta^g(X)\times(\alpha^g(U)\otimes\alpha^g(V)) \ar[r]_{\ \ \ \ \iota\times\alpha^g_2} & \beta^g(X)\times\alpha^g(U\otimes V) \ar[r]_{\beta^g_2} & \beta^g (X\times(U \otimes V))}
$$
$$
\xymatrix@C=0em{
\beta^g(\beta^h(X)\times\alpha^h(U)) \ar[d]_{\beta^g(\beta_2^h)} & & \beta^g\beta^h(X)\times \alpha^g\alpha^h(U) \ar[ll]_{\beta_2^g} \ar[d]^{\eta_\beta^{g,h}\times \eta_\alpha^{g,h}} \\
\beta^g\beta^h(X\times U)  \ar[dr]_{\eta_\beta^{g,h}} & & \beta^{g h}(X)\times\alpha^{g h}(U)\ar[dl]^{\beta_2^{gh}}\\
& \beta^{gh}(X\times U). \\
}
$$
\end{definition}

\begin{remark}
Equivariant module categories have been defined in~\cite{MR2735754}. It is not difficult to check that the above definition is equivalent to the one in~\cite{MR2735754}.
\end{remark}

\begin{example}\label{exmpl:when-D-is-DB}
Let $A$ be a Hopf algebra with an action of $\Gamma$ by Hopf algebra automorphisms $(\alpha_g)_{g\in\Gamma}$. As the category $\CC$ we take the category of left corepresentations (finite dimensional left comodules) $\Corep(A)$. Given a left corepresentation $U=(E_U,\delta^U\colon E_U \to A \otimes E_U)$, the monoidal functor $\alpha^g$ is defined by $E_{\alpha^g (U)} = E_U$ and $\delta^{\alpha^g (U)} = (\alpha^g\otimes\iota) \delta^U$.  Next, let $(B, \delta^B\colon B \to A \otimes B,\beta)$  be a $\Gamma$-$A$-comodule algebra. Then the category $\tilde\DD_B$ of right $B$-modules $M$ with a compatible left $A$-comodule structure $\delta^M$ becomes a right $\Comod(A)$-module category. Concretely, for $M \in \tilde\DD_B$, we take (following the convention of \cite{MR3291643}) $M \times U$ to be $E_U \otimes M$ with the left $A$-coaction given by $v \otimes x \mapsto S(v_{(-1)}) x_{(-1)} \otimes v_{(0)} \otimes x_{(0)}$ and the right $B$-action $(v \otimes x).b = v \otimes x b$.  The induced action of $\Gamma$ on $\tilde\DD_B$ is as follows: $\beta^g(M)$ has the same underlying space as $M$, while the  $A$-coaction becomes $(\alpha_g \otimes \iota) \delta^M$ and the right $B$-module structure is twisted by $\beta_{g^{-1}}$, so that $v.x$ in $\beta^g(M)$ is the same as $v.\beta_{g^{-1}}(x)$ in~$M$. We thus get a $\Gamma$-$\Corep(A)$-module category $\tilde\DD_B$. Note that the action of $\Gamma$ on $\tilde\DD_B$ is strict, and if the tensor product of vector spaces is strict, then the $\Corep(A)$-module structure is also strict. Note also that the trivial case $B = K$ gives (essentially) $\Corep(A)$ as a $\Gamma$-$\Corep(A)$-module category in a natural way.
\end{example}

The following has already been observed by Galindo~\cite{MR2763944}.

\begin{proposition}[\cite{MR2763944}*{Proposition 5.12}]\label{prop:eqv-action-is-crossed-product-action}
Any $\CC\rtimes\Gamma$-module category can be considered as a $\Gamma$-$\CC$-module category. Conversely, the structure of a $\Gamma$-$\CC$-module category extends in an essentially canonical way to that of a $\CC\rtimes\Gamma$-module category.
\end{proposition}

\begin{proof}
Let us sketch an argument.
Suppose we have an action of $\CC \rtimes \Gamma$ on $\DD$. Then we put $\beta^g(X) = X \times (\un_\CC \boxtimes g^{-1})$. The natural isomorphisms $\eta^{g,h}_\beta$ and $\beta^g_2$ are defined using the structure morphisms of $\DD$. Explicitly, $\eta_\beta^{g,h}=\mu^{\un\boxtimes h^{-1},\un\boxtimes g^{-1}}$ and $\beta_2^g$ is given  by the compositions
\begin{multline*}
(X \times g^{-1}) \times \alpha^g(U) \to X \times (g^{-1} \boxtimes \alpha^g(U))=X\times(\alpha^{g^{-1}}\alpha^g(U)\boxtimes g^{-1})\\
\to X\times(U\boxtimes g^{-1}) \to (X\times U )\times g^{-1},
\end{multline*}
where we have abbreviated $U \boxtimes e$ and $\un_\CC \boxtimes g$ as $U$ and $g$ respectively, and the middle equality follows from the way monoidal product in $\CC \rtimes \Gamma$ is defined.

Conversely, given a $\Gamma$-$\CC$-module category $\DD$, we can define $$X \times (U\boxtimes g) = \beta^{g^{-1}}(X \times U),$$ and then extend this definition to direct sums in an essentially unique way. The structure morphisms are defined using those given by the  $\Gamma$-$\CC$-module structure. The axioms are verified by a straightforward but tedious computation.
\end{proof}

From this we obtain the notion of a $\Gamma$-equivariant $\CC$-module functor. It is also possible to formulate it in terms of $\Gamma$-actions and $\CC$-module structures from outset, and then the above argument provides an equivalence of two $2$-categories, that of the $(\CC \rtimes \Gamma)$-module categories (as $0$-cells, module functors as $1$-cells, and natural transformations of module functors as $2$-cells) on the one hand, and that of $\Gamma$-$\CC$-module categories on the other, but we are not going to pursue this.

\smallskip

Let $\CC$ be a $\Gamma$-graded monoidal category equipped with an invariant weak action $\alpha\colon \Gamma \curvearrowright \CC$. Assume $\DD$ is a right $\Gamma$-$\CC$-module category, with the action of $\Gamma$ denoted by $\beta$. Then we can consider~$\DD$ as a $\CC^{t,\alpha}$-module category, since by definition $\CC^{t,\alpha}$ is a subcategory of $\CC\rtimes\Gamma$. We denote the $\CC^{t,\alpha}$-module category $\DD$ by $\DD^{t,\alpha,\beta}$, or by $\DD^{t,\beta}$. As we will see below, this almost tautological definition is the right analogue of the construction in Section~\ref{sec:twist-comod-algebr}.

\subsection{Strictification}
\label{sec:strictification}

Let us show that any $\Gamma$-$\CC$-module structure can be replaced by a strict one in the strongest sense. A part, if not all, of the arguments below should be known to the experts. In fact, as we were finalizing this paper, Galindo's work \cite{arXiv:1604.01679} appeared, which discusses strictification of equivariant monoidal categories and relies on similar ideas.

\smallskip

Given a monoidal category $\CC$, by Mac Lane's theorem there exists a strict monoidal category~$\CC'$ equivalent to $\CC$. Next assume that $\DD$ is a right $\CC$-module category. Since this simply means that we are given a monoidal functor $\CC^{\otimes\mathrm{op}}\to\underline{\End}(\DD)$, $\DD$ can also be considered as a $\CC'$-module category. The following is a folklore result, see, e.g.,~\cite{MR1824165}*{Proposition~5} for a weaker early version.

\begin{proposition}
\label{prop:strict-module}
There exists a strict right $\CC'$-module category $\DD'$ equivalent to $\DD$.
\end{proposition}

\bp
Let $\DD'$ be the category of pairs $(X,U)$, with $X\in\DD$ and $U\in\CC'$. The morphisms are defined by
${\DD'}((X,U),(Y,V))={\DD}(X\times U,Y\times V)$, and the $\CC'$-module structure by $(X,U)\times V=(X,U\otimes V)$. The equivalence $(F,\theta)\colon\DD'\to\DD$ is given by $F(X,U)=X\times U$ and $\theta=\mu^{U,V}_X\colon F(X,U)\times V\to F(X,U\otimes V)$.
\ep

Suppose now that we have a weak action $(\alpha,\eta)$ of $\Gamma$ on a category $\DD$. We may assume that $\alpha^e=\Id_\DD$. Define a new category $\tilde{\DD}$ as follows. The objects are pairs $(X_*, \xi)$, where:
\begin{itemize}
\item $X_*$ is a formal direct sum $\bigoplus_g X_g \boxtimes g$, with $X_g \in \CC$,
\item $\xi$ is a collection of isomorphisms $\xi^h_g\colon \alpha^h (X_g) \to X_{h g}$ making the diagram
\begin{equation}\label{eq:mor-part-in-tDD}
\xymatrix@C=4em{
\alpha^k \alpha^h (X_g) \ar[r]^{\alpha^k(\xi^h_g)} \ar[d]_{\eta^{k,h}} & \alpha^k (X_{h g}) \ar[d]^{\xi^k_{hg}}\\
\alpha^{k h} (X_g) \ar[r]_{\xi^{k h}_g}  & X_{k h g}
}
\end{equation}
commutative.
\end{itemize}
A morphism from $(X, \xi)$ to $(X', \xi')$ is given by a family of morphisms $X_g \to X'_g$ compatible with $\xi$ and $\xi'$ (of course, more succinctly, we could just take the morphisms $X_e \to X'_e$).

The group $\Gamma$ acts strictly on $\tDD$ as follows. For each $h \in \Gamma$, we define an endofunctor $\tilde{\alpha}^h$ of $\tDD$ by sending $\bigoplus_g X_g \boxtimes g$ to $\bigoplus_g X_{g h} \boxtimes g$ and $\xi=(\xi^k_g)_{k,g}$ to $(\xi^k_{gh})_{k,g}$. Since we are only translating the variable on $\Gamma$ on the right, we have the equality $\tilde{\alpha}^h \tilde{\alpha}^k = \tilde{\alpha}^{h k}$.

\begin{proposition}\label{prop:eqv-eqv-D-tilde-D}
The categories $\DD$ and $\tDD$ are $\Gamma$-equivariantly equivalent via the functors
\begin{align*}
F &\colon \DD \to \tDD, \quad X \mapsto \left(\bigoplus_g \alpha^g (X)\boxtimes g, \: \xi^h_g = \eta^{h,g}\colon \alpha^h \alpha^g (X) \to \alpha^{h g} (X)\right)\\
F' &\colon \tDD \to \DD, \quad (\bigoplus_g X_g \boxtimes g, \xi) \to X_e.
\end{align*}
The construction $\DD \rightsquigarrow \tDD$ is natural in the sense that any $\Gamma$-equivariant functor $(G, \theta)$ from $\DD$ to $\DD'$ induces a canonical strict $\Gamma$-equivariant functor $\tDD \to \tDD'$.
\end{proposition}

\begin{proof}
Since $\alpha^e =\Id_\DD$, we have $F' F =\Id_\DD$. On the other hand, we can define a natural isomorphism $F F' \to \Id_{\tDD}$ by combining $\xi^g_e \colon \alpha^g (X_e) =(F F'(X_*, \xi))_g  \to X_g$ for $g \in \Gamma$. Thus, these functors are equivalences of categories.

When $h \in \Gamma$, the required natural isomorphism $\zeta^h \colon F \alpha^h \to \tilde{\alpha}^h F$ is given by collecting $\eta^{g,h} \boxtimes \iota_g \colon \alpha^g \alpha^h (X) \otimes g \to \alpha^{g h} (X) \boxtimes g$ for $g \in \Gamma$. Commutativity of the diagram~\eqref{eq:eqv-ftr} follows from $\eta^{g, h k} \alpha^g(\eta^{h,k}) = \eta^{g h, k} \eta^{g h}$.

\smallskip

Consider now a $\Gamma$-equivariant functor $(G, \theta)$ from $\DD\to\DD'$.
Let $(X_*,\xi) = (\bigoplus_g X_g \boxtimes g, (\xi^h_g)_{h,g})$ be an object in $\tDD$. Then we define an object $\tilde{G}(X_*,\xi)$ of $\tDD'$ to be the pair consisting of $Y_* = \bigoplus_g G(X_g) \boxtimes g$ and the isomorphisms $$\nu^h_g = G(\xi^h_g) \theta^h_{X_g}\colon \beta^h G(X_g) \to G(X_{hg}).$$ That the family $(\nu^h_g)_{h,g}$ satisfies the commutativity of~\eqref{eq:mor-part-in-tDD} follows from the corresponding condition for $\xi$ and the commutativity of~\eqref{eq:eqv-ftr}.
\end{proof}

\begin{remark}
The above argument is inspired by the work of Tambara~\cite{MR1815142}. For finite groups the category~$\tDD$ is $(\DD \boxtimes \mathrm{Vect}^\Gamma_{f,K})^\Gamma$ (see \cite{MR1815142} or the next subsection for the meaning of this notation; here $\mathrm{Vect}^\Gamma_{f,K}$ is the category of finite dimensional $\Gamma$-graded vector spaces over $K$), and the equivalence $\DD \simeq (\DD \boxtimes \mathrm{Vect}^\Gamma_{f,K})^\Gamma$ appears in the proof of~Theorem~4.1 in~\textit{op.~cit}.

The conclusion of the above proposition might look counter-intuitive at first, since on $\DD = \mathrm{Vect}_{f,K}$ (considered as a $K$-linear category, so without the monoidal structure), the weak actions of $\Gamma$ are parametrized by $H^2(\Gamma; K^\times)$ up to equivalence (see, e.g.,~\cite{MR3242743}*{Exercise~2.7.3}). However, strictness is not preserved under the natural correspondence of weak actions on equivalent categories, so we are not claiming that all weak actions on $\DD$ can be simultaneously strictified on some equivalent category.
\end{remark}

Let us next consider a monoidal category $\CC$ and a weak action $(\alpha,\eta)$ of $\Gamma$ on $\CC$. Then the category $\tCC$ defined as above admits the structure of a monoidal category. Namely, the tensor product of $(X_*, \xi)$ and $(Z_*, \zeta)$ is defined to be the pair consisting of the object $\bigoplus_g (X_g \otimes Z_g) \boxtimes g$ and the family of morphisms
$$
(\xi\otimes\zeta)^h_g = (\xi^h_g \otimes \zeta^h_g)(\alpha^h_2)^{-1}\colon \alpha^h(X_g \otimes Z_g) \to X_{h g} \otimes Z_{h g}.
$$
If $\CC$ is strict, then the tensor product in $\tCC$ is also strict. If we further assume that conditions~\eqref{eq:weakact2}--\eqref{eq:weakact3} are satisfied, then the tensor unit $(\bigoplus_g \un\boxtimes g, (\xi^h_g=\iota)_{h,g})$ in $\tilde\CC$ is also strict, so that $\tCC$ becomes a strict monoidal category. Then $\tilde\alpha^g$ becomes a strict tensor functor, and we get the following.

\begin{lemma}\label{lem:mon-eqv-C-tilde-C}
Given a weak action of $\Gamma$ on a strict monoidal category $\CC$ satisfying conditions~\eqref{eq:weakact1}--\eqref{eq:weakact3}, the functors $F\colon \CC \to \tCC$ and $F'\colon \tCC \to \CC$ of Proposition~\ref{prop:eqv-eqv-D-tilde-D} can be enriched to $\Gamma$-equivariant monoidal equivalences.
\end{lemma}

\begin{proof}
The functor $F'$ is already a strict tensor functor, so we only have to define $F_2$. Given~$X$ and~$Z$ in $\CC$, we define the natural transformation $F_2\colon F (X) \otimes F (Z) \to F(X \otimes Z)$ by the collection of morphisms $\alpha^g_2 \colon \alpha^g (X) \otimes \alpha^g (Z) \to \alpha^g(X \otimes Z)$. The fact that this gives morphisms in $\tCC$ follows from monoidality of $\eta^{g,h}$.  For the same reason the natural isomorphisms $\zeta^h$ from the proof of  Proposition~\ref{prop:eqv-eqv-D-tilde-D} are monoidal.
\end{proof}

Let us summarize the above considerations. We start with a $\Gamma$-$\CC$-module category $\DD$ and perform the following steps:
\begin{itemize}
\item take a strict monoidal category $\CC'$ equivalent to $\CC$; since $\underline{\Aut}^\otimes(\CC')$ is monoidally equivalent to $\underline{\Aut}^\otimes(\CC)$, we have a weak action of $\Gamma$ on $\CC'$ such that the monoidal equivalence between $\CC$ and $\CC'$ becomes $\Gamma$-equivariant;
\item next we modify the action on $\CC'$ to get an isomorphic action satisfying conditions~\eqref{eq:weakact1}--\eqref{eq:weakact3};
\item we then apply to $\CC'$ the procedure described before Proposition~\ref{prop:eqv-eqv-D-tilde-D} and Lemma~\ref{lem:mon-eqv-C-tilde-C} to get a strict monoidal category $\tilde\CC$ equipped with a strict action of $\Gamma$ by strict tensor functors such that $\tCC\simeq\CC'$ as $\Gamma$-equivariant monoidal categories;
\item as $\tCC\rtimes\Gamma$ is monoidally equivalent to $\CC\rtimes\Gamma$, we can consider $\DD$ as a $\tCC\rtimes\Gamma$-module category; applying Proposition~\ref{prop:strict-module} we finally get an equivalent strict $\tCC\rtimes\Gamma$-category $\DD'$.
\end{itemize}

Therefore in developing a general theory of $\Gamma$-equivariant module categories there is no loss of generality in assuming that for every such category everything that can be strict is strict.

Note, however, that we still need nonstrict functors between such categories. But if we consider only singly generated categories, as we do below, even this is unnecessary.

\subsection{Duality}

We want to show next that the constructions of the algebras $R^{t,\beta}$ and module categories $\DD^{t,\beta}$ given in Sections~\ref{ssec:twist-comod} and~\ref{ssec:twist-mod-cat}, respectively, are related by a Tannaka--Krein type duality.

The Tannaka--Krein duality principle states that quantum groups can be encoded by monoidal categories and fiber functors into the category of vector spaces. Pursuing this correspondence, the $A$-comodule algebras are encoded by~$\Corep(A)$-module categories, see~\citelist{\cite{MR1976459}\cite{MR3121622}\cite{MR3426224}\cite{MR3291643}} for precise statements in various contexts. Since for infinite dimensional Hopf algebras complete results of this type seem to be available only for compact quantum groups, we will work now in this setting.

Specifically, we need the following Tannaka--Krein type duality result for compact quantum groups $G$ and unital $G$-C$^*$-algebras.

\begin{theorem}[\citelist{\cite{MR3121622}*{Theorem~6.4}\cite{MR3426224}*{Theorem~3.3}}]
\label{tactions}
Let $G$ be a reduced compact quantum group.  Then the following two categories are equivalent:
\begin{itemize}
\item the category of unital $G$-C$^*$-algebras $B$ with unital $G$-equivariant $*$-homomorphisms as morphisms;
\item the category of pairs $(\DD,X)$, where $\DD$ is a right $(\Rep G)$-module C$^*$-category closed under subobjects and $X$ is a generating object in $\DD$, with equivalence classes of unitary $(\Rep G)$-module functors respecting the prescribed generating objects as morphisms.
\end{itemize}
\end{theorem}

Here, by a generating object~$X$ in a $\CC$-module category $\DD$ we mean that every object of $\DD$ is a subobject of $X\times U$ for some~$U$. Having generating objects $X \in \DD$ and $X' \in \DD'$ specified, we only consider the module functors $(F, \theta)$ such that $F(X) = X'$. The equivalence relation on such functors $(F, \theta), (F', \theta')\colon \DD \to \DD'$ is defined as the existence of a natural unitary transformation $\eta\colon F \to F'$ which satisfies $\eta_X = \iota$ and is compatible with module functor structures $\theta$ and $\theta'$, in the sense that the diagrams of the following form are commutative:
$$
\xymatrix{
F(Y) \times U \ar[r]^{\theta} \ar[d]_{\eta \times \iota} & F(Y \times U) \ar[d]^{\eta}\\
F'(Y) \times U \ar[r]_{\theta'} & F'(Y \times U).
}
$$

\smallskip

Briefly, the correspondence in the above theorem is defined as follows. Given a unital $G$-C$^*$-algebra $B$, as $\DD_B$ we take category of $G$-equivariant finitely generated right Hilbert $B$-modules (which is a full subcategory of the category $\tilde\DD_B$ considered in Example~\ref{exmpl:when-D-is-DB}). The distinguished object $X$ is $B$ itself.

In the opposite direction, given a pair $(\DD,X)$ as in the theorem, the corresponding $G$-C$^*$-algebra $A_{\DD,X}$ is defined as a suitable completion of the \emph{regular algebra}
$$
\mathcal{A}_{\DD,X} = \bigoplus_{[U] \in \Irr G} \bar{H}_U \otimes \DD(X, X \times U),
$$
where $\bar{H}_U$ is the conjugate Hilbert space of $H_U$, endowed with the product structure coming from the monoidal structure on $\Rep G$ and the $(\Rep G)$-module category structure on $\DD$. In more detail,  the product of $\bar\xi\otimes S\in \bar{H}_U \otimes \DD(X, X \times U)$ and $\bar\zeta\otimes T\in \bar{H}_V \otimes \DD(X, X \times V)$ is obtained by writing
$$
\overline{(\xi\otimes\zeta)}\otimes\mu^{U,V}(S\times\iota)T\in \bar{H}_{U\otimes V} \otimes \DD(X, X \times (U\otimes V))
$$ as an element of $\mathcal{A}_{\DD,X}$ using a decomposition of $U\otimes V$ into irreducibles.

\smallskip

Turning to equivariant module categories, recall the following notion.

\begin{definition}[\cite{MR1815142}]\label{dfn:gamma-fixed-obj}
Given a weak action $(\beta,\eta_\beta)$ of $\Gamma$ on a category $\DD$, a \emph{$\Gamma$-invariant}, or $\Gamma$-equivariant, object is a pair $(X,\zeta)$ consisting of an object $X\in\DD$ and a family $\zeta$ of isomorphisms $\zeta^g\colon \beta^g(X) \to X$ such that the diagrams
 $$
 \xymatrix@C=4em{
 \beta^g \beta^h (X) \ar[r]^{\beta^g(\zeta^h)} \ar[d]_{\eta_\beta^{g,h}} & \beta^g(X) \ar[d]^{\zeta^g}\\
 \beta^{g h}(X) \ar[r]_{\zeta^{g h}} & X
 }
 $$
 commute for $g, h \in \Gamma$. The category of $\Gamma$-invariant objects is denoted by $\DD^\Gamma$.
\end{definition}

In the C$^*$-setting we of course also assume that the isomorphisms $\zeta^g$ are unitary.

\begin{remark}
If $\DD$ is a $\Gamma$-$\CC$-module category, then any invariant object $(X,\zeta)$ defines a $\Gamma$-$\CC$-module functor $\CC\to\DD$ mapping $\un$ into $X$.
\end{remark}

Let us say that $\Gamma$ \emph{fixes} $X\in\DD$ if $(X,\zeta)\in\DD^\Gamma$ for some $\zeta$. We have the following result complementing Theorem~\ref{tactions}.

\begin{proposition}
Let $G$ be a reduced compact quantum group endowed with an action $\alpha$ of~$\Gamma$, $\DD$ a right $(\Rep G)$-module C$^*$-category closed under subobjects, and $X\in\DD$ be a generating object. Then the following conditions are equivalent:
\begin{enumerate}
\item there is an action of $\Gamma$ on $A_{\DD,X}$ turning this algebra into a $\Gamma$-equivariant $G$-C$^*$-algebra;
\item the $(\Rep G)$-module category structure on $\DD$ extends to that of a right $\Gamma$-$(\Rep G)$-module C$^*$-category such that $\Gamma$ fixes $X$.
\end{enumerate}
\end{proposition}

\begin{proof}
If $B=A_{\DD,X}$, then by Theorem~\ref{tactions} we may assume that $\DD$ is the category $\DD_B$ introduced above and $X$ is $B$ itself.
Assume we are given an action $\beta$ of $\Gamma$ on $B$ turning it into a $\Gamma$-equivariant $G$-C$^*$-algebra. Then by Example~\ref{exmpl:when-D-is-DB} we can define a strict action of $\Gamma$ on $\DD_B$, so that $\DD_B$ becomes a $\Gamma$-$(\Rep G)$-module category. Taking as $\zeta^g\colon\beta^g(X)\to X$ the maps $\beta_g\colon B\to B$, we see that $(X,\zeta)\in\DD^\Gamma_B$.

Conversely, suppose that $\DD$ is a $\Gamma$-$(\Rep G)$-module C$^*$-category, with the action of $\Gamma$ on $\DD$ denoted by $\beta$, and $X$ is fixed by $\Gamma$. By Proposition~\ref{prop:strict-module} we may assume that the $(\Rep G)\rtimes\Gamma$-module category $\DD$ is strict, so that both the action of $\Gamma$ and the $(\Rep G)$-module structure on~$\DD$ are strict and the isomorphisms $\beta^g_2\colon\beta^g(Y)\times\alpha^g(U)\to\beta^g(Y\times U)$ are the identities.

 Choose $\zeta$ such that $(X,\zeta)\in\DD^\Gamma$. Note that for every irreducible representation $U$ of $G$, the space $\bar H_U\otimes\DD(X,X\times U)$ can be identified with a subspace of $\mathcal{A}_{\DD, X}$ in a canonical way, no matter which representatives of irreducible representations were used to define $\mathcal{A}_{\DD, X}$. We then define a map $\beta_g\colon\mathcal{A}_{\DD, X}\to \mathcal{A}_{\DD, X}$ by
$$
\bar H_U\otimes\DD(X,X\times U)\ni\bar\xi\otimes T\mapsto \bar\xi\otimes(\zeta^g\times\iota)\beta^g(T)(\zeta^g)^{-1}
\in \bar H_{\alpha^g(U)}\otimes\DD(X,X\times \alpha^g(U)).
$$
It is easy to see that this is a $*$-automorphism. The cocycle identity $\zeta^{gh}=\zeta^g\beta^g(\zeta^h)$ implies that $\beta_{gh}=\beta_g\beta_h$, so we get an action of $\Gamma$ on $A_{\DD,X}$. Finally, it is again immediate by definition that the action of $G$ on $A_{\DD,X}$ is $\Gamma$-equivariant.
\end{proof}

\begin{remark}\label{rem:how-to-get-crossed-prod-comod-alg}
Assume $(\DD,(X,\zeta))$ is as in the above proof. If we consider $\DD$ as a $(\Rep G)\rtimes\Gamma$-module and carry out the reconstruction for $(\DD,X)$, then we obtain the $C(G)\rtimes_{ \alpha, r}\Gamma$-comodule algebra $A_{\DD,X}\rtimes_{\beta, r} \Gamma$. If the $(\Rep G)\rtimes\Gamma$-module category structure is strict, the canonical unitaries $u_g\in A_{\DD,X}\rtimes_{\beta,r} \Gamma$ are given by $u_g=\bar 1\otimes(\zeta^{g^{-1}})^{-1}\in \bar\C\otimes\DD(X,\beta^{g^{-1}}(X))=\bar \C\otimes\DD(X,X\times g)$.
\end{remark}

Recall that the universal grading group of a monoidal category $\CC$ is called the \emph{chain group} and denoted by~$\Ch(\CC)$~\citelist{\cite{MR2097019}\cite{MR2130607}} (see also~\cite{MR733774}). For $\CC = \Rep G$, we also denote this group by $\Ch(G)$.

We are now ready to establish a duality between the constructions of $R^{t,\beta}$ and $\DD^{t,\beta}$.

\begin{proposition}
Let $G$ be a reduced compact quantum group, $\Gamma$ a quotient of $\Ch(G)$, and $\alpha$ an invariant action of $\Gamma$ on $G$ with respect to the associated $\Gamma$-grading. Let $\DD$ be a $\Gamma$-$(\Rep G)$-module C$^*$-category, with the action of $\Gamma$ denoted by $\beta$, and $(X,\zeta)$ be an object in $\DD^\Gamma$. Consider the corresponding $\Gamma$-equivariant $G$-C$^*$-algebra $A_{\DD,X}$, and denote the action of $\Gamma$ again by $\beta$.
Then the $G^{t,\alpha}$-C$^*$-algebra~$A_{\DD^{t,\beta},X}$ associated with the $(\Rep G^{t,\alpha})$-module category $\DD^{t,\beta}$ and the object $X$ is isomorphic to the closure $(A_{\DD, X})^{t,\beta}$ of $(\A_{\DD,X})^{t,\beta}$ in $A_{\DD, X} \rtimes_{\beta,r}  \Gamma$.
\end{proposition}

\begin{proof}
As linear spaces, $(\A_{\DD, X})^{t,\beta}$ and $\A_{\DD^{t,\beta},X}$ can be identified in a straightforward way. Moreover, since $\Pol(G^{t,\alpha})$ can be regarded as a sub-$*$-bialgebra of $\Pol(G) \rtimes_{\alpha} \Gamma$, the algebra $\A_{\DD^{t,\beta},X}$ embeds into $A_{\DD \curvearrowleft \Rep G \rtimes_\alpha \Gamma, X}$, which is the crossed product $A_{\DD, X} \rtimes_{\beta,r}  \Gamma$ by Remark~\ref{rem:how-to-get-crossed-prod-comod-alg}. Combining these observations, we obtain that $A_{\DD^{t,\beta},X}$ can be indeed identified with the `diagonal' subalgebra of $A_{\DD, X} \rtimes_{\beta,r}  \Gamma$.
\end{proof}

\section{Poisson boundaries of twisted categories}
\label{sec:poisson-boundary}

In this final section we study the relation between the \emph{categorical Poisson boundaries} and graded twisting. We follow the conventions and terminology of~\cite{arXiv:1405.6572}. In particular, we assume that~$\CC$ is a rigid C$^*$-tensor category, closed under finite direct sums and subobjects, with simple unit and (at most) countable number of irreducible classes.

Fix representatives $\{U_s\}_s$ of isomorphism classes of simple objects in $\CC$. Denote by $m^s_{rt}$ the multiplicity of $U_s$ in $U_r\otimes U_t$. More generally, for any object $X$ we denote by $m^s_{Xt}$ the multiplicity of $U_s$ in $X\otimes U_t$ and write $\Gamma_X$ for the matrix  with entries $a_{st}=m^s_{Xt}$.

For each $s$, consider the Markov operator
$$
P_s\colon \ell^\infty(\Irr(\CC)) \to \ell^\infty(\Irr(\CC)), \quad (P_s f)(r) = \sum_{t \in X} p_s(r, t) f(t)
$$
with transition probabilities
$$
p_s(r,t)=m^t_{sr}\frac{d(t)}{d(s)d(r)},
$$
where we write $d(s)$ for the intrinsic dimension $d(U_s)$ of $U_s$.
For an arbitrary probability measure~$\mu$ on $\Irr(\CC)$ we put $P_\mu=\sum_s\mu(s)P_s$.

Recall that $\CC$ is said to be \emph{weakly amenable} if there is a left-invariant mean on $\ell^\infty( \Irr (\CC))$, that is, a state invariant under the operators $P_s$. This is equivalent to existence of an \emph{ergodic} probability measure $\mu$, meaning that the only $P_\mu$-invariant functions in $\ell^\infty(\Irr(\CC))$ are scalars. Recall also that $\CC$ is called \emph{amenable} if $d(X)=\|\Gamma_X\|$ for all $X$.

The next proposition was already proved~\cite{MR3580173} when $\CC$ is of the form $\Rep G$ for some coamenable compact quantum group $G$.

\begin{proposition}
 Suppose that $\CC$ is weakly amenable. Then its chain group $\Ch(\CC)$ is amenable.
\end{proposition}

\begin{proof}
From the amenability of the Poisson boundary for an ergodic measure~\cite{arXiv:1405.6572}, we know that the map $X\mapsto \|\Gamma_X\|$ is a dimension function, which we denote by $d_a$. Thus, we obtain a semiring homomorphism from $\Z_{\ge 0}[\Irr(\CC)]$ to $\R_{\ge 0}$ by sending $[X]$ to $d_a(X)$.

Consider the convolution of measures on $\Irr(\CC)$ defined using the dimension function $d_a$, so
$$
(\nu*\mu)(t)=\sum_{s,r}\nu(s)\mu(r)m^t_{sr}\frac{d_a(t)}{d_a(s) d_a(r)}.
$$
We write $\mu^n$ instead of $\mu^{*n}$. Let $\mu$ be any nondegenerate probability measure on~$\Irr(\CC)$ which is symmetric, i.e., $\mu(\bar{s}) = \mu(s)$ for $s \in \Irr(\CC)$, with $U_{\bar{s}} \simeq \bar{U}_s$. One of equivalent ways of expressing the amenability of $d_a$~\cite{MR1644299}*{Section~4} is
$$
\lim_{n\to\infty}\sqrt[2n]{\mu^{2n}([\un])} = 1.
$$
Define a probability measure on $\Ch(\CC)$ by $\tilde{\mu}(g) = \sum_{[U]\in\Irr(\CC)\colon g_U=g} \mu([U])$, where $g_U\in\Ch(\CC)$ is the degree of a simple object $U$. We have
$$
\widetilde{\mu*\nu}=\sum_{[U], [V] \in \Irr(\CC), W\subset U\otimes V}\mu([U])\nu([V])\frac{d_a( W)}{d_a (U) d_a (V)} \delta_{g_W},
$$
where $W$ runs through a maximal family of mutually orthogonal simple summands of $U \otimes V$. Since for such $W$ we have $g_W=g_Ug_V$ by definition of the chain group, and since
$$
\sum_{W\subset U\otimes V} d_a (W) = d_a (U) d_a( V),
$$
we get
$$
\widetilde{\mu*\nu}=\sum_{[U],[V]} \mu([U])\nu([V]) \delta_{g_U g_V} = \tilde{\mu} * \tilde{\nu}.
$$
It follows that $\tilde{\mu}^{2n}(e) \ge \mu^{2n}([\un])$, which implies the amenability of $\Ch(\CC)$.
\end{proof}

\begin{proposition}\label{prop:wk-amen-perm-cross-prod}
Suppose that $\Ch(\CC)\to \Gamma$ is a surjective homomorphism and $\alpha\colon\Gamma \curvearrowright \CC$ is an invariant weak action on $\CC$ with respect to the associated $\Gamma$-grading. Then
\begin{enumerate}
\item if $\CC$ is weakly amenable, then $\CC^{t,\alpha}$ and $\CC\rtimes\Gamma$ are weakly amenable;
\item if $\CC$ is amenable, then $\CC^{t,\alpha}$ and $\CC\rtimes\Gamma$ are amenable.
\end{enumerate}

\end{proposition}

\begin{proof} (i) By the previous proposition we already know that $\Gamma$ is amenable. Let $m_\CC$ be a left invariant mean on $\ell^\infty (\Irr(\CC))$, and $m_\Gamma$ be a left invariant mean on $\ell^\infty( \Gamma)$. The group $\Gamma$ acts on~$\Irr(\CC)$ and we may assume that $m_\CC$ is invariant under this action. Indeed, for $f\in\ell^\infty(\Irr(\CC))$ and $g\in\Gamma$ put $f_g=f(g^{-1}\cdot)$. Define a state $\tilde m_\CC$ on $\ell^{\infty}(\Irr(\CC))$ by
$$
\tilde m_\CC=m_\Gamma(g\mapsto m_\CC(f_g)).
$$
This state is $\Gamma$-invariant, so $\tilde m_\CC(f_g)=\tilde m_\CC(f)$, and it is still an invariant mean, which can be easily seen by checking first that $P_s(f)_g=P_{g s}(f_g)$ for all $s\in\Irr(\CC)$ and $g\in\Gamma$.

Now, in order to prove that $\CC^{t,\alpha}$ is weakly amenable, let us identify $\Irr(\CC^{t,\alpha})$ with $\Irr(\CC)$. We denote by $P^{t,\alpha}_s$ the Markov operators on $\ell^\infty(\Irr(\CC))=\ell^\infty(\Irr(\CC^{t,\alpha}))$ defined by $\CC^{t,\alpha}$. Then if $s\in\Irr(\CC)$ has degree $g$, we have
\begin{equation*}
P^{t,\alpha}_s(f)=P_{s}(f_{g^{-1}}).
\end{equation*}
As $m_\CC$ is $\Gamma$-invariant, from this we see that $m_\CC$ is a left-invariant mean for $\CC^{t,\alpha}$.

It is even easier to see that if we identify $\Irr(\CC\rtimes\Gamma)$ with $\Irr(\CC)\times\Gamma$, then the formula
$$
f\mapsto m_\Gamma(g\mapsto m_\CC(f(\cdot,g)))
$$
defines a left-invariant mean for $\CC\rtimes\Gamma$.

\smallskip

(ii) We already know that $\CC\rtimes\Gamma$ is weakly amenable. Hence $X\mapsto\|\Gamma_X\|$ is a dimension function on $\CC\rtimes\Gamma$. Since $\CC$ is a full rigid monoidal subcategory of $\CC\rtimes\Gamma$, we have $\|\Gamma_U\| = \|\Gamma_{U\boxtimes e}\|$. Moreover, since $\un\boxtimes g$ is invertible, we have
$$
\| \Gamma_{U\boxtimes g}\| =  \|\Gamma_{U\boxtimes e}\| = \|\Gamma_U\|.
$$

On the other hand, the intrinsic dimension satisfies $d(U\boxtimes g) = d(U)$. Combining these observations, we obtain
$$
\| \Gamma_{\bigoplus_i U_i \boxtimes g_i} \| = \sum_i \|\Gamma_{U_i}\| = \sum_i d(U_i) = \sum_i d(U_i \boxtimes g_i).
$$
This shows the amenability of $\CC \rtimes \Gamma$. As $\CC^{t,\alpha}$ is a full subcategory of $\CC\rtimes\Gamma$, it is also amenable.
\end{proof}

We next want to compare the Poisson boundaries of twisted categories. As above, assume that a rigid C$^*$-tensor category $\CC$ is $\Gamma$-graded using a surjective homomorphism $\Ch(\CC)\to\Gamma$, and that we are given an invariant action $\alpha$ of $\Gamma$ on $\CC$. By the considerations in Section~\ref{sec:strictification}, we may assume that both $\CC$ and the action $\alpha$ are strict, and that $\alpha^g$ is a strict monoidal autoequivalence for all $g$.

Consider the category $\hat\CC$ as in~\cite{arXiv:1405.6572}, which is the C$^*$-tensor category (with nonsimple unit) with the same objects as in $\CC$ and the morphisms $U\to V$ given by bounded natural transformations between the endofunctors $\iota\otimes U\colon X \mapsto X \otimes U$ and $\iota\otimes V$ on $\CC$. Then the category $\hat\CC$ is still $\Gamma$-graded and we have an invariant action of $\Gamma$ on it, so we can form a graded twisting ${\hat\CC\,}^{t,\alpha}$. Define a functor
$F\colon {\hat\CC\,}^{t,\alpha}\to\widehat{\CC^{t,\alpha}}$ as follows. On the one hand we set $F(U\boxtimes g)=U\boxtimes g$ for the objects. On the other hand, take a morphism $\eta\in {\hat\CC\,}^{t,\alpha}(U\boxtimes g,V\boxtimes g)$. By definition it has the form $\tilde\eta\boxtimes\iota$, with $\tilde\eta\in\hat\CC(U,V)$. We then define $F(\eta)\in \widehat{\CC^{t,\alpha}}(U\boxtimes g,V\boxtimes g)$ as the unique morphism such that
\begin{equation}\label{eq:twist-C-hat-functor}
F(\eta)_{X\boxtimes h}=\alpha^h(\tilde\eta_{\alpha^{h^{-1}}(X)})\boxtimes\iota\colon (X\otimes\alpha^h(U))\boxtimes h g \to (X\otimes\alpha^h(V))\boxtimes h g
\end{equation}
for all $h\in\Gamma$ and $X\in\CC_h$.

\begin{lemma}
The functor $F \colon {\hat\CC\,}^{t,\alpha}\to\widehat{\CC^{t,\alpha}}$ is a unitary strict monoidal isomorphism of categories.
\end{lemma}

\bp It is immediate that $F$ is bijective on morphisms and objects and that it is $*$-preserving. So we only need to check that it is a strict tensor functor. This is a routine verification. Let us check, for example, that $F(\iota_{Y\boxtimes k}\otimes\eta)=\iota_{Y\boxtimes k}\otimes F(\eta)$. Using the notation before the lemma, we have
$$
(\iota_{Y\boxtimes k}\otimes F(\eta))_{X\boxtimes h}=F(\eta)_{(X\boxtimes h)\otimes(Y\boxtimes k)}=F(\eta)_{(X\otimes\alpha^h(Y))\boxtimes hk}
=\alpha^{hk}(\tilde\eta_{\alpha^{(hk)^{-1}}(X\otimes\alpha^h(Y))})\boxtimes\iota.
$$
On the other hand,
\begin{align*}
F(\iota_{Y\boxtimes k}\otimes\eta)_{X\boxtimes h}&=F((\iota_Y\otimes\alpha^k(\tilde\eta))\boxtimes\iota)_{X\boxtimes h}
=\alpha^{h}((\iota_Y\otimes\alpha^k(\tilde\eta))_{\alpha^{h^{-1}}(X)})\boxtimes\iota\\
&=\alpha^{h}(\alpha^k(\tilde\eta)_{\alpha^{h^{-1}}(X)\otimes Y})\boxtimes\iota=\alpha^{hk}(\tilde\eta_{\alpha^{k^{-1}}(\alpha^{h^{-1}}(X)\otimes Y)}))\boxtimes\iota,
\end{align*}
and we see that we get the desired equality.
\ep

Define Markov operators $P_s$ on $\hat\CC(U,V)$ by
$$
P_s(\eta)_X=(\tr_s\otimes\iota)(\eta_{U_s\otimes X}),
$$
where $\tr_s\otimes\iota$ is the normalized categorical partial trace.
When $U=V=\un$, so that $\hat\CC(U,V)$ can be identified with $\ell^\infty(\Irr(\CC))$, these are the same operators~$P_s$ that we introduced earlier.

\begin{definition}
We say that a natural transformation $\eta\in\hat\CC(U,V)$ is \emph{absolutely harmonic} if $P_s(\eta)=\eta$ for all $s\in\Irr(\CC)$.
\end{definition}

We will also say that a morphism in $\hat\CC\rtimes\Gamma$ is absolutely harmonic if its homogeneous components are absolutely harmonic.

Denote by $\PP(U,V)\subset\hat\CC(U,V)$ the subspace of absolutely harmonic elements. For $U=V$ this is an ultraweakly closed operator subspace of the von Neumann algebra $\hat\CC(U)$, so it admits at most one structure of a C$^*$-algebra. If $\CC$ is weakly amenable, such a structure indeed exists by results of~\cite{arXiv:1405.6572}. By considering $\PP(U,V)$ as a subspace of $\PP(U\oplus V)$ we then get a composition rule for all absolutely harmonic elements, so $\PP$ becomes a C$^*$-tensor category. We then complete this category with respect to finite direct sums and subobjects and continue to denote this completion by $\PP$. This category together with the embedding functor $\CC\to\PP$ is the Poisson boundary of $\CC$ with respect to any ergodic measure~\cite{arXiv:1405.6572}. We will therefore call it the \emph{absolute Poisson boundary}. When needed, we will write $\PP(\CC)$ instead of $\PP$.

\begin{theorem}\label{thm:twist-poisson}
The strict tensor functor $F\colon  {\hat\CC\,}^{t,\alpha}\to\widehat{\CC^{t,\alpha}}$ defined by~\eqref{eq:twist-C-hat-functor} maps the absolutely harmonic elements onto the absolutely harmonic ones. In particular, if $\CC$ is weakly amenable, then $F$ defines an isomorphism between $\CC^{t,\alpha}\to\PP(\CC)^{t,\alpha}$ and the absolute Poisson boundary of~$\CC^{t,\alpha}$.
\end{theorem}

\bp
Take $\eta=\tilde\eta\boxtimes\iota\in {\hat\CC\,}^{t,\alpha}(U\boxtimes g,V\boxtimes g)=\hat\CC(U,V)\boxtimes\iota$. We have to show that $\tilde\eta$ is absolutely harmonic if and only if $F(\eta)$ is absolutely harmonic. For this, take $Y=U_s$ and let $k$ be the degree of $Y$. We will write $P^{t,\alpha}_s$ for the Markov operator on $\widehat{\CC^{t,\alpha}}$ defined by $Y\boxtimes k$. Then for $X\in\CC_h$ we have
$$
P^{t,\alpha}_s(F(\eta))_{X\boxtimes h}=(\tr_{Y\boxtimes k}\otimes \iota)(F(\eta)_{(Y\boxtimes k)\otimes(X\boxtimes h)})
=(\tr_{Y\boxtimes k}\otimes \iota)(\alpha^{kh}(\tilde\eta_{\alpha^{(kh)^{-1}}(Y\otimes\alpha^k(X))})\boxtimes\iota).
$$
Now observe that if $T=\tilde T\boxtimes\iota\in\CC^{t,\alpha}((Y\boxtimes k)\otimes(W\boxtimes l))=\CC(Y\otimes\alpha^k(W))\boxtimes\iota$, then
$$
(\tr_{Y\boxtimes k}\otimes\iota)(T)=\alpha^{k^{-1}}(\tr_Y\otimes\iota)(\tilde T)\boxtimes\iota.
$$
It follows that
\begin{align*}
P^{t,\alpha}_s(F(\eta))_{X\boxtimes h}
&=\alpha^{k^{-1}}(\tr_{Y}\otimes \iota)(\alpha^{kh}(\tilde\eta_{\alpha^{(kh)^{-1}}(Y\otimes\alpha^k(X))}))\boxtimes\iota\\
&=\alpha^{h}(\tr_{\alpha^{(kh)^{-1}}(Y)}\otimes \iota)(\alpha^{kh}(\tilde\eta_{\alpha^{(kh)^{-1}}(Y)\otimes\alpha^{h^{-1}}(X))}))\boxtimes\iota\\
&=\alpha^h(P_{(kh)^{-1}s}(\tilde\eta)_{\alpha^{h^{-1}}(X)})\boxtimes\iota.
\end{align*}
As $F(\eta)_{X\boxtimes h}=\alpha^h(\tilde\eta_{\alpha^{h^{-1}}(X)})\boxtimes\iota$, we therefore see that $F(\eta)$ is absolutely harmonic if and only~if
$$
\tilde\eta_{\alpha^{h^{-1}}(X)}=P_{(kh)^{-1}s}(\tilde\eta)_{\alpha^{h^{-1}}(X)}
$$
for all $h,k\in\Gamma$, $X\in\CC_h$ and $s\in\Irr(\CC)$ of degree $k$. As the action of $\Gamma$ preserves the degree, this is equivalent to absolute harmonicity of $\eta$.

\smallskip

Assume now that $\CC$ is weakly amenable. By Proposition~\ref{prop:wk-amen-perm-cross-prod}(i) we know that $\CC^{t,\alpha}$ is also weakly amenable, so it has a well-defined absolute Poisson boundary. As $F$ defines a complete order isomorphism between the spaces of absolutely harmonic elements, it respects the composition rule for such elements, hence it defines an isomorphism between $\CC^{t,\alpha}\to\PP(\CC)^{t,\alpha}$ and the absolute Poisson boundary of~$\CC^{t,\alpha}$.
\ep

\begin{remark}
If $G$ is a coamenable compact quantum group, then it is known that the absolute Poisson boundary of $\Rep G$ is the forgetful functor $\Rep G\to\Rep K$, where $K\subset G$ is the maximal quantum subgroup of Kac type~\citelist{\cite{MR2335776}\cite{MR3291643}}. Applying the above theorem in this case we conclude that the maximal quantum subgroup of $G^{t,\alpha}$ of Kac type is $K^{t,\bar\alpha}$, where $\bar\alpha$ is the action of~$\Gamma$ on $K$ induced by~$\alpha$. This conclusion, however, is true for any $G$ and it does not require the Poisson boundary theory. Indeed, one of the equivalent ways of defining $\Pol(K)$ is as the quotient of $\Pol(G)$ by the ideal $I$ generated by the elements $a-S^2(a)$. This ideal is $\Gamma$-graded and $\Gamma$-stable, so, as discussed in Section~\ref{ssec:quotients}, we get a Hopf ideal $j(I)\subset\Pol(G^{t,\alpha})$. By working with homogeneous components it is easy to see that $j(I)$ is again the ideal generated by $a-S^2(a)$, so for the maximal quantum subgroup $\tilde K\subset G^{t,\alpha}$ of Kac type we get $\Pol(\tilde K)=\Pol(G^{t,\alpha})/j(I)=(\Pol(G)/I)^{t,\bar\alpha}=\Pol(K^{t,\bar\alpha})$.
\end{remark}

For completeness, let us also briefly consider Poisson boundaries of crossed products. For $g\in\Gamma$ define a map
$F_g\colon \hat\CC(U,V)\to \widehat{\CC\rtimes\Gamma}(U\boxtimes g, V\boxtimes g)$ by
$$
F_g(\eta)_{X\boxtimes h}=\alpha^h(\eta_{\alpha^{h^{-1}}(X)})\boxtimes\iota\colon (X\otimes\alpha^h(U))\boxtimes hg\to (X\otimes\alpha^h(V))\boxtimes hg.
$$
Using these maps we easily get the following result.

\begin{proposition}
For any $g\in\Gamma$ and $U,V\in\CC_g$, the map $F_g$ defines a bijection between the absolutely harmonic elements in $\hat\CC(U,V)$ and in $\widehat{\CC\rtimes\Gamma}(U\boxtimes g, V\boxtimes g)$. Furthermore, if $\CC$ is weakly amenable (so that $\CC\rtimes\Gamma$ is also weakly amenable), then $F_e$ defines a unitary monoidal equivalence between $\PP(\CC)$ and the full subcategory of $\PP(\CC\rtimes\Gamma)$ formed by the objects $X\boxtimes e$, which then extends to a unitary monoidal equivalence between $\PP(\CC)\rtimes\Gamma$ and $\PP(\CC\rtimes\Gamma)$.
\end{proposition}

\begin{remark}
The equivalence between $\PP(\CC)\rtimes\Gamma$ and $\PP(\CC\rtimes\Gamma)$ can be established even without the minimal amount of computations indicated above. Namely, one of the main properties of the absolute Poisson boundary $\CC\to\PP(\CC)$ is that this is a universal \emph{amenable} functor~\cite{arXiv:1405.6572}. Therefore, by this universality, the natural embedding $\CC \to \CC \rtimes \Gamma \to \PP(\CC\rtimes\Gamma)$ induces an essentially unique tensor functor $\PP(\CC) \to \PP(\CC\rtimes\Gamma)$ (which is of course the functor~$F_e$). We also have a canonical embedding of $\underline{\Gamma}$ into $\PP(\CC\rtimes\Gamma)$. These two embeddings combine to give a tensor functor $\PP(\CC)\rtimes \Gamma \to \PP(\CC\rtimes\Gamma)$.  Conversely, by Proposition~\ref{prop:wk-amen-perm-cross-prod}, $\PP(\CC)\rtimes \Gamma$ is amenable. Moreover, the embedding $\CC \to \PP(\CC) \rtimes \Gamma$ induces a tensor functor $\CC\rtimes\Gamma \to \PP(\CC)\rtimes\Gamma$. Again by the universality, we obtain a tensor functor $\PP(\CC \rtimes \Gamma) \to \PP(\CC)\rtimes \Gamma$. The tensor functors between $\PP(\CC\rtimes \Gamma)$ and $\PP(\CC)\rtimes\Gamma$ are identities on $\CC$ and $\underline{\Gamma}$, hence by the universality they are equivalences.

When $\Gamma$ is abelian, similar arguments provide an alternative route to Theorem~\ref{thm:twist-poisson}. Namely, we have an embedding of $\CC^{t,\alpha}$ into $\PP(\CC)^{t,\alpha}$, which induces a unitary tensor functor $\Xi_\alpha\colon\PP(\CC^{t,\alpha})\to\PP(\CC)^{t,\alpha}$. Using the commutativity of $\Gamma$, we can `untwist' $\CC^{t,\alpha}$ and go back to $\CC$. Then the universality implies that $\Xi_\alpha$ is a unitary monoidal equivalence.
\end{remark}

\end{document}